\documentclass[12pt,a4paper]{amsart}
\usepackage[utf8]{inputenc}
\usepackage{amsmath, amssymb, amsfonts, amsthm}
\usepackage{graphicx} 
\usepackage{xcolor}
\usepackage{hyperref}
\usepackage{geometry}
\usepackage{quiver}
\usepackage[shortlabels]{enumitem}

\theoremstyle{plain}
\newtheorem*{thm*}{Theorem}
\newtheorem{thm}{Theorem}[section]
\newtheorem*{lemma*}{Lemma}
\newtheorem{lemma}[thm]{Lemma}
\newtheorem{pro}[thm]{Proposition}
\newtheorem*{pro*}{Proposition}
\newtheorem{cor}[thm]{Corollary}
\newtheorem*{cor*}{Corollary}
\newtheorem*{con*}{Question}
\newtheorem{con}[thm]{Question}

\theoremstyle{definition}
\newtheorem{df}[thm]{Definition}
\newtheorem*{df*}{Definition}

\theoremstyle{remark}
\newtheorem{rem}[thm]{Remark}
\newtheorem*{rem*}{Remark}
\newtheorem{ex}[thm]{Example}
\newtheorem*{ex*}{Example}

\newcommand{\multi}[1]{{\underline{\eta}\phantom{}_{#1}}} 
\newcommand{\mzeta}{{\underline{\zeta}}} 
\newcommand{\cvar}{{c_\bullet}} 
\newcommand{\svar}{{s_\bullet}} 
\newcommand{\tvar}{{t_\bullet}} 
\newcommand{\tvarr}{{\underline{t}}} 
\newcommand{\xvar}{{\underline{x}}} 
\newcommand{\avar}{{\underline{a}}} 
\newcommand{\rvar}{{\underline{r}}} 
\newcommand{\ThT}[1]{{\operatorname{Th}}^T_{#1}} 
\newcommand{\Thom}[1]{{\operatorname{Th}}_{#1}} 
\newcommand{\KazT}[1]{S_{#1}} 
\newcommand{\Slocus}[2]{\Sigma^{S}_{#1}(#2)} 
\newcommand{\Tlocus}[2]{\Sigma^{T}_{#1}(#2)} 
\newcommand{\Qlocus}[2]{\Sigma_{#1}(#2)} 
\DeclareMathOperator{\codim}{codim}
\DeclareMathOperator{\scodim}{scodim}
\DeclareMathOperator{\tcodim}{tcodim}

\newcommand{\Strato}[2]{\Hilb_{#1}(#2)}
\newcommand{\Strat}[2]{\overline{\Hilb_{#1}(#2)}}
\newcommand{\StratF}[2]{\mathcal{H}\hspace{-0.25ex}\mathit{ilb\/}_{#1}(#2)}%

\newcommand{\Aut}{\operatorname{Aut}}

\newcommand{\coh}{\operatorname{H}}
\newcommand{\id}{\operatorname{id}}

\newcommand{\Hilb}{{\operatorname{Hilb}}}
\newcommand{\supp}{\operatorname{supp}}
\newcommand{\Spec}{\operatorname{Spec}}
\newcommand{\Der}{\operatorname{Der}}
\newcommand{\Hom}{\operatorname{Hom}}
\newcommand{\III}{I\!I\!I}

\newcommand{\PP}{\mathcal{P}}

\newcommand\K{\mathcal{K}}
\newcommand\A{\mathcal{A}}
\newcommand\E{\mathcal{E}}
\def\O{\mathcal{O}}

\newcommand{\TT}{\mathbb{T}}

\newcommand{\NN}{\mathbb{N}}
\newcommand{\QQ}{\mathbb{Q}}
\newcommand{\CC}{\mathbb{C}}
\newcommand{\Aa}{\mathbb{A}}

\newcommand{\m}{\mathfrak{m}}

\newcommand{\kk}{\Bbbk}

\newcommand\proto[2]{\theta_{#1}^{#2}}

\newcommand{\T}{\Spec A}
\newcommand{\G}{G}
\newcommand{\B}{\mathcal{B}}
\newcommand{\GL}{\operatorname{GL}}
\newcommand{\Iso}{\operatorname{Isom}_S}
\newcommand{\onto}{\twoheadrightarrow}

\title{Degenerations of multisingularities and Artin algebras}

\author{Jakub Koncki}
\address{Institute of Mathematics, University of Warsaw, Warsaw, Poland}

\author{Rich\'ard Rim\'anyi}
\address{Department of Mathematics, UNC Chapel Hill, Chapel Hill, NC, USA}
\curraddr{HUN-REN R\'enyi Institute of Mathematics, Budapest, Hungary}

\begin{document}

\begin{abstract}
We study the degeneration hierarchy of commutative, associative, finite-dimensional complex Artin algebras. Instead of studying degenerations in the Hilbert scheme, we introduce a singularity-theoretic notion of degeneration based on the correspondence between singularities of stable map germs and local algebras. This leads to a natural partially ordered set, the {\em stable hierarchy}, in which one algebra degenerates to another if nearby singularities realize the latter. 

Our first main result is that, in a wide range of dimensions, this hierarchy can be determined purely from symmetry data, namely from the automorphism groups of the algebras or singularities. The key tool is the theory of certain equivariant characteristic classes called Thom polynomials of multisingularities, established by Kazarian. Suitable substitutions into these polynomials completely characterize the hierarchy. As a consequence, the computation of degeneration posets becomes algorithmic in nature. 

In our second main result, we prove that our singularity-theoretic hierarchy extends the algebraic hierarchy obtained from deformation theory. While deformation theory requires the dimension (rank) to be fixed, our hierarchy generalizes this framework by comparing algebras of varying dimensions.

\end{abstract}

\maketitle

\section{Introduction}

Consider commutative, associative, finite-dimensional complex Artin algebras. In small dimensions there are finitely many isomorphism types, while starting in dimension $7$ continuous families (moduli) appear.

In this paper we study the degeneration hierarchy of such algebras using methods from { singularity theory} and {equivariant geometry}.

Our first key idea is to define the hierarchy via singularity theory, rather than via degenerations in the Hilbert scheme. The main tool is the relationship between sin\-gu\-la\-ri\-ties and local algebras, or equivalently, between multisingularities and Artin algebras, which we recall in Section~\ref{sec:singularity}. Informally, we declare that an algebra $Q$ degenerates to an algebra $R$ if, for any stable (multi)singularity with local algebra $R$, there exist arbitrarily nearby points with (multi)singularities whose local algebra is $Q$ (see Figure~\ref{fig:Whitney_umbrella}). The resulting partially ordered set will be called the stable hierarchy.

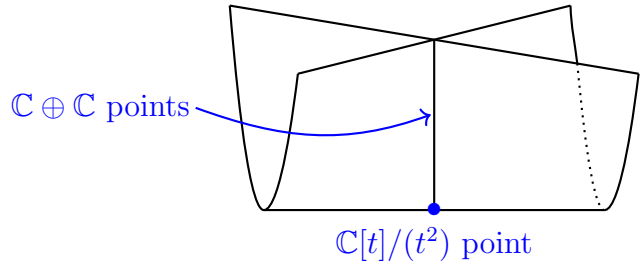
\begin{figure}
\[
    \begin{tikzpicture}[scale=.45, baseline=10]
        \draw[thick] (-5,0) to (5,0);
        \draw[thick] (-6,6) parabola bend (-5,0) (-4,4);
        \draw[thick] (4,6) to[out=-88,in=97] (4.2,4.3);
        \draw[thick] (5,0) parabola bend (5,0) (6,4);
        \draw[thick, dotted] (5,0) parabola bend (5,0) (4.2,4.3);
        \draw[thick] (-6,6) -- (6,4);
        \draw[thick] (-4,4)-- (4,6);
        \draw[thick] (0,0) -- (0,5);
          \draw[thick, blue,->] (-7,3) to[out=-20,in=200]
          (-0.1,2.8);
          \draw[blue] (-9.8,3) node {$\CC\oplus \CC$ points};
          \draw[blue] (0,-1) node {$\CC[t]/(t^2)$ point};
           \draw[blue] (0,0) node {$\bullet$};
    \end{tikzpicture}
    \]
    \caption{The so-called Whitney umbrella singularity $(x,y)\mapsto (x^2,xy,y)$ is stable. It shows that the algebra $Q=\CC \oplus \CC$ degenerates to the algebra $R=\CC[t]/(t^2)$.}
\label{fig:Whitney_umbrella}
\end{figure}

The second key idea is that this hierarchy can be determined purely from symmetry data, namely from the automorphism groups of singularities and algebras. Within the framework of equivariant geometry, this information allows one to compute certain cohomological characteristic classes---Thom polynomials of multisingularities. The general theory of such polynomials was pioneered by Ronga, Herbert, Kleiman, Piene, and it was established by Kazarian \cite{KazaMulti}; for recent advances see  \cite{BSzmult,  TOmult, JKRRmult}. We prove that, in a wide range of dimensions, the vanishing of appropriate substitutions into these polynomials completely determines the hierarchy.

As a consequence, computing poset structures such as those in Figures~\ref{fig:galaxy}, \ref{fig:Hierarchy_of_mono_6}, and \ref{fig:splits} becomes, in principle, a matter of computational power. This constitutes our first main result. 

While the foundational and computational techniques in Thom polynomial theory are widely studied 
by experts 
of global singularity theory, 
proving `folklore' statements with precision (like Proposition~\ref{cor:Der}), as well as coding the algorithms to fast computer algebra systems are other key contributions of our work.

\medskip
Our singularity theory version of algebra-hierarchy is closely related to degenerations. In Section~\ref{sec:comparison} we prove our second main result stating that, when restricted to algebras of fixed dimension (rank), the stable hierarchy coincides with the classical degeneration partial order arising from deformation theory. For instance, Figure \ref{fig:Hierarchy_of_mono_6} depicts degenerations of local algebras of dimension 6. This comparison makes it possible to apply techniques from singularity theory in deformation theory and vice versa. In particular, degenerations of algebras of dimension at most 6 can be determined algorithmically using Thom polynomials. 

To prove the relation between the singularity theory version of hierarchy and degenerations of algebras, we use a stratification of the Hilbert scheme of points by isomorphism type of algebra. This stratification and its properties are 
known to experts, yet we were unable to find a reference in the literature. For completeness, we include an Appendix~\ref{s:deformation} with a definition of a stratum in terms of a representable functor, together with basic properties of the stratification. 

\medskip
\noindent\emph{A note to experts.} Equivariant geometry is a powerful tool for studying the hierarchy of strata in a stratification. One may ask why we do not apply these methods directly to the Hilbert scheme, the natural parameter space of Artin algebras. The key advantage of the singularity-theoretic approach is that the ambient space is a vector space, whose equivariant cohomology ring is a polynomial ring. This stands in sharp contrast to the notoriously complicated local structure of Hilbert schemes, e.g. \cite{JJPathologies,Szach, JJPath2,Oszer}. The price to pay for this simplification is that we must work with \emph{stable} singularities, and hence possibly with high degree polynomials in many variables.

\medskip

\medskip


Our study is also motivated by these connections: (1) the notion of a \emph{concise tensor} is closely related to Artin algebras, and the corresponding hierarchy problem is relevant in geometric complexity theory  \cite{JJtensors}; (2) \emph{lattice homology}, a powerful tool in the study of 3-manifolds, knots, numerical semigroups, and monomial ideals, has recently been applied to Artin algebras, where deformation theory plays a central role \cite{nemethi}.

\bigskip

Throughout the paper, the base field is $\CC$, and we use cohomology with rational coefficients. For maps and germs we use the terminology “source” and “target” in place of “domain” and “codomain.” The \emph{relative dimension} of a map or germ from dimension $m$ to $n$ is defined as $n-m$, and is denoted by $l$.

\bigskip
\noindent\textbf{Acknowledgments.} This work would not have been possible without the generous support,  insight, and constant help of Joachim Jelisiejew. 
The first author was supported by National Science Centre (Poland) grant Sonata 2025/59/D/ST1/00012.
The second author was supported by the U.S. National Science Foundation
under Grant No. 2152309. Any opinions, findings, and conclusions or recommendations expressed in this
material are those of the author(s) and do not necessarily reflect the views of the NSF. We are grateful to Piotr Oszer and Andrzej Weber for discussions on the topic.

\section{Local algebras of small dimension}
\label{sec:algebras}
Consider commutative, associative, finite dimensional, local $\CC$-algebras. Examples include 
\[
A_n=\CC[x]/(x^{n+1}),
\qquad
I_{ab}=\CC[x,y]/(x^a+y^b,xy),
\qquad
\III_{ab}=\CC[x,y]/(x^a,xy,y^b).
\]
One approach to classify such algebras is according to their dimensions, see e.g. \cite{damon_algebras, Mazzola, poonen, yu, becker}. The outcome of these classification efforts (with occasional slight corrections and additions) is available on \cite{TPP}, up to dimension 7. Here is a brief summary. 

The complete list of algebras of dimension 1, 2, 3, 4 is
\[
A_0,
\qquad 
A_1, 
\qquad
A_2, \III_{22},
\qquad
A_3, I_{22}, \III_{23}, C_3:=\CC[x,y,z]/(x,y,z)^2.
\]
The complete list of 5-dimensional algebras is
\begin{center}
\begin{tabular}{ll}
$A_4, I_{23}, \III_{24}, \III_{33}$, &
$C_{41}:=\CC[x,y,z]/(xy,xz,yz,x^2-y^2,y^2-z^2)$,
\\
$B_4:=\CC[x,y]/(x^2,xy^2,y^3)$, &
$C_{42}:=\CC[x,y,z]/(x^2,y^2,z^2,xz,yz)$, \\
$D_4:=\CC[x,y,z,w]/(x,y,z,w)^2$, & $C_{43}:=\CC[x,y,z]/(x^2,y^2,z^3,xy,xz,yz)$.
\end{tabular}
\end{center}
In Figure \ref{Fig:List of 6dim algebras} we list the 25 isomorphism types of algebras of dimension 6.

Starting in dimension 7, the classification includes moduli. For example, in dimension~7 there are two  one-dimensional families of algebras (corresponding to ``general nets of conics'' and ``general pencils of quadrics''), and another 77 discrete types (or a few more if we distinguish some special ones in the families) \cite{TPP}.

A finite subset of algebras (each of dimension $\leq 9$) relevant in singularity theory are called Mather algebras, cf. Definition~\ref{def:Mather_singularity}.

\begin{rem}
In classification theorems one needs to go through various cases carefully, and it is natural to make omissions or mistakes. However, in Section~\ref{sec:classification_vs_identity}, we present an argument that can be interpreted as a {\em numerical verification} of our classification statements.
\end{rem}

\begin{rem}
Another approach to the classification of algebras---not needed in the present paper---is according to their so-called Thom-Boardman type. The classification of algebras with Thom-Boardman types $\Sigma^0$, $\Sigma^{1,1,\ldots, 1,0}$ and $\Sigma^{2,0}$ is discrete, and results for other types are in \cite{dPW}. 
\end{rem}

\section{Singularity Theory}
\label{sec:singularity}
The singularity theory of mappings studies maps between manifolds by examining them locally. Here we present a brief introduction; for a detailed treatment, see the monographs \cite{Golub_book,Varchenko_book, MNB}.
\subsection{Contact singularities}
Let $m$ and $l$ be natural numbers. We denote by $\E(m,m+l)$ the space of  germs $(\CC^m,0)\to (\CC^{m+l},0)$ in the analytic topology. This space carries a natural action of the group of germs of biholomorphic maps
$$\A=\operatorname{Diff}(\CC^m,0)\times\operatorname{Diff}(\CC^{m+l},0)\,.$$
Let $f:M^m\to N^{m+l}$ be an algebraic map between smooth varieties of dimension $m$ and $m+l$.
To a point $x\in M$ we associate a germ
$\eta_x\in\E(m,m+l)$ by choosing local charts around $x$ and $f(x)$. This germ is not uniquely defined as it depends on the choice of charts. Different choices produce germs that differ by the action of the group $\A$. This motivates the following definition.
\begin{df}
    An $\A$-singularity is an orbit of the group $\A$ in $\E(m,m+l)$.
\end{df}
The $\A$-singularities form a stratification of the germ space. It is often convenient to consider a coarser stratification. To a germ $f\in\E(m,m+l)$ we associate its local algebra
\begin{align} \label{w:local_alg}
    Q_f=\CC[[x_1,\dots,x_m]]/(f_1,\dots,f_{m+l})
\end{align}
and stratify the space $\E(m,m+l)$ according to the isomorphism type of the algebra $Q_f$.
\begin{df}
    The contact singularity corresponding to a local algebra $Q$ is the subset of $\E(m,m+l)$ defined by
    \begin{equation}\label{eq:eta_Q}
    \eta_Q=\eta=\{
    f\in \E(m,m+l)| Q_f\simeq Q
    \}\,.
    \end{equation}
\end{df}

\begin{figure}
\begin{center}
\begin{tabular}{llll}
$A_5$ & $=(x^6)$ &   $C_{51}$ & $= (x^2+y^2+z^2,xy,xz,yz) $  \\
& & $C_{52}$ & $= (x^2,y^2,z^2,xy+xz) $\\
$I_{24}$ & $=(x^2+y^4,xy)$ & $C_{53}$ & $= (x^2,xy+z^3,y^2,xz,yz) $\\
$I_{33}$ & $=(x^3+y^3,xy)$ & $C_{54}$ & $= (x^2+yz,xz,y^2,z^2) $\\
$\III_{25}$ & $=(x^2,xy,y^5)$ & $C_{55}$ & $= (xy,xz,y^2,z^2,x^3) $ \\
$\III_{34}$ & $=(x^3,xy,y^4)$ & $C_{56}$ & $= (y^2+x^3,z^2,xy,xz,yz) $\\
& & $C_{57}$ & $= (xy,yz,z^2,y^2-xz,x^3) $ \\
$B_{51}$ & $= (x^2,y^3)$  & $C_{58}$ & $= (x^2,xy,y^2,z^2) $ \\
$B_{52}$ & $= (x^2+y^3,xy^2,y^4) $ & $C_{59}$ & $= (x^2,xy,y^2,xz,yz,z^4) $ \\
$B_{53}$ & $= (x^2,xy^2,y^4) $ & $C_{5,10}$ & $= (x^2,xy,xz,yz,y^3,z^3) $ \\
$B_{54}$ & $= (x,y)^3=(x^3,x^2y,xy^2,y^3) $ & $C_{5,11}$ & $= (x^2,xy,xz,y^2,yz^2,z^3) $ 
\end{tabular}
\end{center}
\begin{center}
\begin{tabular}{ll}
$D_{51}$ & $= (x^2,y^2,z^2,w^2,xy-zw,xz,xw,yz,yw) $ \\
$D_{52}$ & $= (x^2,y^2+z^2,y^2+w^2,xy,xz,xw,yz,yw,zw)  $ \\
$D_{53}$ & $= (x^2,y^2,z^2,w^2,xy,xz,xw,yz,yw)  $ \\
$D_{54}$ & $= (x^2,y^2,z^2,xy,xz,xw,yz,yw,zw,w^3)  $ \\
$E_5$ & $= (x,y,z,w,v)^2=(x^2,y^2,z^2,v^2,w^2,xy,\dots,wv)$
\end{tabular}
\end{center}
\caption{Complete list of 6-dimensional algebras. For brevity we only indicate the ideal in the polynomial ring generated by the obvious generators. Cf. Figure~\ref{fig:Hierarchy_of_mono_6}.}
\label{Fig:List of 6dim algebras}
\end{figure}

 A contact singularity is an $\A$-invariant subset and hence a union of $\A$-singularities. Later, we restrict our attention to well-behaved germs, called stable germs. A contact singularity contains either no stable germs or an open dense $\A$-orbit of stable germs. Consequently, the study of stable $\A$-singularities is equivalent to the study of stable contact singularities.
 \begin{rem} \label{rem:K}
     Contact singularities can also be described as orbits of a certain group 
     action on the space of germs, for details see \cite[Sect. 4]{MNB}.
 \end{rem}
 It is standard practice to identify contact singularities corresponding to the same algebra $Q$ and number $l$ but to different values of $m$. This approach is motivated by a study of the suspension map
 $$\sigma:\E(m,m+l)\to \E(m+1,m+1+l)$$
 defined by $\sigma(f)=f\times \id$. The local algebras of a germ $f$ and its suspension $\sigma(f)$ coincide. In particular, the image of any contact singularity is contained in a contact singularity. Moreover, the suspension map $\sigma$ is transversal to all contact singularities, cf. \cite[Lem. 7.8]{FRannals}.

\begin{rem}
    In practice, we can replace the vector space $\E(m,m+l)$ of germs, and the group $\A$ with their $N$-jets ($N \gg 0$) to obtain $\E_N(m,m+l)$ and $\A_N$. In this way we obtain an algebraic group acting on a finite dimensional vector space. Our constructions and results do not depend on $N$ as long as $N$ is large enough, hence by abuse of notation we will not write the subscript~$N$.
\end{rem}
In this paper, we deal mostly with contact singularities. We omit the word contact in the notation, referring to them simply as singularities.
\subsection{Multisingularities}
Singularities are a tool to describe local properties of a map. They tell us how the map behaves near a point of the source. However, they are not enough to detect target-local structures, such as multiple points. To approach this problem, one studies multisingularities.
\begin{df}  
     A multisingularity $\multi{}$ is a finite multiset of singularities.  We will use intuitive notation, for example, if $\eta_1$ and $\eta_2$ denote singularities, then $(\eta_1^5,\eta_2^2)$ or simply $\eta_1^5\eta_2^2$ will denote the multiset containing these two singularities with multiplicities 5 and~2.
\end{df}
\begin{df}
    For a nonempty multisingularity $\multi{}=(\eta_1^{a_1},\dots,\eta_k^{a_k})$ define
    $$|\Aut(\multi{})|=a_1!\cdot a_2!\cdot\ldots\cdot a_k!\,.$$
\end{df}
    The empty set is a multisingularity. A singularity can be regarded as multisingularity of cardinality one.
\begin{rem}
    A multisingularity is a multiset, not an ordered sequence, even if we often write $\multi{}=(\eta_1,\dots,\eta_k)$ tacitly introducing an order.
\end{rem}

Contact singularities correspond to local algebras. Multisingularities can be identified with products of such algebras. In particular every finite dimensional algebra determines a multisingularity.
\begin{df}
    Let $Q$ be a finite dimensional algebra. Then $Q$ is isomorphic to a product of local algebras
    $$Q\simeq Q_1\times Q_2\times\dots\times Q_k\,.$$
    The multisingularity corresponding to the algebra $Q$ is the multiset $\multi{}=(\eta_1,\dots,\eta_k)$, where $\eta_i$ is the multisingularity corresponding to the local algebra $Q_i$.
\end{df}
A multisingularity corresponds to a point in the target of a map (cf. Definition~\ref{def:loci} below). There is also a notion of multisingularity corresponding to a point in the target of a map with a chosen preimage.
\begin{df}
    An S-multisingularity (source-multisingularity) is a multisingularity $\multi{}$ with a distinguished element, i.e. a pair $(\eta,\multi{})$ such that $\eta$ is a singularity, $\multi{}$ is a multisingularity and $\eta\in \multi{}$.
\end{df}
    We presented a contact version of the definition of a multisingularity. There is an alternative approach using multigerms. The space of multigerms
    $$\bigsqcup_{i=1}^k (\CC^m,0)\to(\CC^{m+l},0)$$
    admits a natural action of the group $\A'=\Sigma_k\ltimes\operatorname{Diff}(\CC^m,0)^k\times\operatorname{Diff}(\CC^{m+l},0)$, where $\Sigma_k$ denotes the permutation group. We say that two germs are $\A$-equivalent if they lie in the same orbit. A multisingularity may be interpreted as an $\A'$-invariant subset of multigerms with a given multiset of local algebras. 

\subsection{Stable germs and maps}
Standard constructions in singularity theory, such as universal counting formulas, are expected to only be valid for maps that are stable under perturbations. Intuitively, we require that any map $g$ closeby to $f$ is equivalent to it: $g= \phi \circ f\circ \psi^{-1}$ for diffeomophisms $\psi, \phi$ of the domain and co-domain. This condition is reasonable over the reals, but not over the complexes (think of, for example, maps from a compact manifold to $\CC^n$). Experience shows that for complex maps the `target-local version' of perturbation and stability are the right substitutes:

\begin{df}[Def. 3.4 in \cite{MNB}]
 An unfolding of a map germ $g:(\CC^m,S)\to (\CC^{m+l},0)$ is a germ $G:(\CC^{m}\times \CC^d , S \times \{0\}) \to (\CC^{m+l} \times \CC^d,0)$ of the form $(\tilde{g}(x,u),u)$ with $\tilde{g}(x,0)=g$. The unfolding $G=g \times \id_{\CC^d}$ is called trivial. If all unfoldings of a germ are equivalent (via the natural equivalence on unfoldings) to a trivial unfolding, then it is called stable.
\end{df}

\begin{df}
A finite, algebraic map $f:M^m \to N^{m+l}$ is called stable if all its induced germs $f: (M,f^{-1}(y)) \to (N,y)$ are 
stable. 
\end{df}
\begin{rem} \label{rem:closed}
    By definition, stable maps are finite, thus closed. In particular, for any $X\subset M$ and $Y\subset N$ we have
    $$
    f\Big(\overline{X}\Big)=\overline{f(X)}\,,\qquad f^{-1}\Big(\overline{Y}\Big)=\overline{f^{-1}(Y)}\,.
    $$
\end{rem}

\begin{df}
  Let $Q$ be a local algebra, $l$ a natural number and $\eta$ the corresponding singularity. We say that $Q$ (or $\eta$) occurs for $l$ if the singularity $\eta$ contains a stable germ, i.e. if for some $m\in \NN$ there exists a stable germ $f\in \E(m,m+l)$ whose local algebra is isomorphic to $Q$. \\
We say that multisingularity occurs for $l$ if all its elements occur for $l$.  
\end{df}

\begin{ex}
     The algebra $\CC[x,y,z]/(x,y,z)^2$ occurs for $l\ge 3$. 
\end{ex}
\begin{pro}
    Suppose that a multisingularity $\multi{}$ occurs for some $l_0$. Then it occurs for all $l\ge l_0$.
\end{pro}

The reason behind the statement in the example and the proposition above is a construction called {\em miniversal unfolding} standard in singularity theory.

\subsection{Singularity loci} 
\label{sec:singular_loci}
In the whole paper we consider only algebraic finite maps between smooth varieties.
Let $f:M^m\to N^{m+l}$ be such a map. As mentioned earlier, for any $x\in M$ we denote by $\eta_x$ a germ in $\E(m,m+l)$ induced by $f$ at $x$. By a slight abuse of notation, we also  write $\eta_x$ for the contact singularity containing this germ. To a point $y\in N$ we associate the multisingularity $\multi{y}$ defined as the {\em multiset}
$$\multi{y}=\{\eta_x\,|\,x\in f^{-1}(y)\}\,.$$
\begin{df} \label{def:loci}
    Let $f:M\to N$ be a finite map between smooth varieties. Let $\eta$ be a singularity and $\multi{}$ a multisingularity containing $\eta$. We define
    \begin{itemize}
        \item  $\eta$-singularity locus in the source
		$$
		\Qlocus{\multi{}}{f}=\{x\in M|\eta_x=\eta\}\subseteq M\,.
		$$
		\item $\multi{}$-multisingularity locus in the target
		$$
		\Tlocus{\multi{}}{f}=\{y\in N|\multi{y}=\multi{}\}\subseteq N\,.
		$$
		\item $(\eta,\multi{})$-multisingularity locus in the source: 
		$$
		\Slocus{\eta,\multi{}}{f}=\{x\in M|\eta_x=\eta\,, \multi{f(x)}=\multi{}\}\subseteq M\,.
		$$
	\end{itemize}
\end{df}
We refer to all these subsets as singularity loci. The following statement is a direct consequence of the definition.
\begin{pro}\label{pro:loci}
    Let $f:M\to N$ be a finite map between smooth varieties. For a multisingularity $\multi{}$ and $\eta\in\multi{} $ we have
    \begin{align*}
        f\big(\Slocus{\eta,\multi{}}{f}\big)&=\Tlocus{\multi{}}{f}
        \,,&
        \Slocus{(\eta,\multi{})}{f}&=f^{-1}\big(\Tlocus{\multi{}}{f}\big)\cap \Qlocus{\eta}{f}\,, \\
        f\big(\Qlocus{\eta}{f}\big)&=\bigcup_{\eta\in\mzeta}\Tlocus{\mzeta}{f}
        \,,&
        \Qlocus{\eta}{f}&=\bigcup_{\eta\in\mzeta}\Slocus{\eta,\mzeta}{f}\,.
    \end{align*}
\end{pro}
\begin{rem} \label{rem:germs}
The loci $\Tlocus{\multi{}}{f}$ form a stratification of the target variety $N$. Reasoning similar to \cite[p.~14]{MNB} shows that the germ of this stratification at a point $y\in N$ depends only on the germ $f:(M,f^{-1}(y))\to(N,y)$.
\end{rem}

In singularity theory, the local algebra is usually defined in terms of formal power series, cf. Formula \eqref{w:local_alg}. Instead of power series one may consider the polynomial ring localized at the maximal ideal $(x_1,\dots,x_m)$.
For germs of finite maps, the algebras obtained by these two constructions coincide. This is a consequence of the following standard algebraic fact.
\begin{pro}
	Let $\CC[x_1,\dots,x_m]$ be a polynomial ring and $\m=(x_1,\dots,x_m)$ the maximal ideal at $0$. Let $f_1,\dots,f_{n}\in \m$ be polynomials such that the algebra $\CC[x_1,\dots,x_m]$ $/$ $(f_1,\dots,f_n)$ is a finite dimensional vector space. Then
	$$
	\frac{\CC[x_1,\dots,x_m]_\m}{(f_1,\dots,f_n)} \simeq \frac{\CC[[x_1,\dots,x_m]]}{(f_1,\dots,f_n)}\,.
	$$
\end{pro}
\begin{cor} \label{cor:formal_local}
	Let $f:M\to N$ be a finite map between smooth varieties. Let $y\in N$ and $F_y\subset M$ be the scheme theoretical fiber of $f$ over $y$. \begin{itemize}
	    \item  For $x\in F_y$ the singularity $\eta_x$ corresponds to the stalk of the sheaf $\O_{F_{y}}$ at $x$. 
        \item The multisingularity $\multi{y}$ corresponds to the coordinate ring $\O_{F_{y}}(F_y)$.
	\end{itemize}
\end{cor}

\subsection{Prototypes}
For a nonempty multisingularity that occurs for a natural number $l$ there exists a universal stable germ belonging to it, see \cite[Sect. 9]{Varchenko_book}. 
\begin{df} \label{df:proto_universal}
    Let $l$ be a natural number and $\multi{} =(\eta_1,\dots,\eta_k)$ a multisingularity that occurs for $l$. There exists a stable germ
    $$
    \proto{\multi{}}{l}:\bigsqcup_{i=1}^k (\CC^m,0)\to (\CC^{m+l},0)\,,
    $$
    called prototype of $\multi{}$ for $l$, such that every stable germ $f:\bigsqcup (\CC^n,0)\to (\CC^{n+l},0)$ lying in the singularity $\multi{}$
    is $\A$-equivalent to $\proto{\multi{}}{l}\times\id_{\CC^{n-m}}$, cf. \cite[Thm. on p. 162]{Varchenko_book}.
\end{df}
\begin{ex}
    Let $\eta$ be the singularity corresponding to the algebra $\CC$. The prototype $\proto{\eta}{l}$ is the inclusion of the origin into $\CC^l$. \\
    Let $\zeta$ be the singularity corresponding to the algebra $\CC[x]/(x^2)$. The prototype $\proto{\zeta}{l}$ for $l=0$ and $l=1$ are of the form
    $$\proto{\zeta}{0}(x)=x^2\,,\qquad \proto{\zeta}{1}(x,y)=(x^2,xy,y) \,.$$
    See \cite[Sect. 2.4]{JKRRmult} and \cite[Sect. 4.3]{MNB} for more examples and a general algorithm of constructing prototypes.
\end{ex}
 When the multisingularity $\multi{}$ consists of a single element $\eta$, we call the map $\proto{\eta}{l}$ the prototype of the singularity $\eta$. It follows from \cite[Thm. 3.3]{MNB} that prototypes of singularities can be used to construct prototypes of multisingularities:

\begin{pro} \label{pro:proto_multi}
	Let $\multi{}=(\eta_1,\dots,\eta_k)$ be a multisingularity and $\proto{\eta_i}{l}: V_i \to W_i$ be the prototype of $\eta_i$. The prototype $\proto{\multi{}}{l}$ is given by
	$$\proto{\multi{}}{l}=\bigsqcup_{i=1}^k \proto{\eta_i}{l} \times \id_{\prod_{j\neq i} W_j}
    :\bigsqcup_{i=1}^k \big(V_i\times \prod_{j\neq i} W_j\big) \to  \prod_{i=1}^k W_i\,.$$
\end{pro}
\begin{ex}
    Let $l=1$. Let $A_0^3$ be the multisingularity corresponding to the algebra $\CC\times\CC\times\CC$. The prototype 
\[
\proto{A_0^3}{1}\colon (\CC^2_{x_1,y_1},0) \sqcup (\CC^2_{x_2,y_2},0) \sqcup (\CC^2_{x_3,y_3},0) \to (\CC^3,0),
\]
is given by 
\[
  (x_1,y_1)\mapsto (0,x_1,y_1), \ 
  (x_2,y_2)\mapsto (x_2,0,y_2), \ 
  (x_3,y_3)\mapsto (x_3,y_3,0).   
\]
\end{ex}
The following standard fact is crucial for the computation of Thom polynomials, see Section \ref{s:Thom}.
\begin{pro} \label{pro:proto_0}
	Let $\proto{\multi{}}{l}:\bigsqcup\CC^m\to \CC^{m+l}$ be the prototype of a multisigularity $\multi{}$ for $l$. We have $\Tlocus{\multi{}}{\proto{\multi{}}{l}}=\{0\}$.
\end{pro}
The existence of prototypes allows the local study of singularity loci.
\begin{pro} \label{pro:local}
	Let $\multi{}$ and $\mzeta$ be multisingularities that occur for $l\in\NN$. Let $\proto{\multi{}}{l}$ be the prototype of $\multi{}$ for $l$ and $f:M\to N$ a stable map of relative dimension $l$ . Suppose that $y\in \Tlocus{\multi{}}{f}$. Then
    $$y\in \overline{\Tlocus{\mzeta}{f}}\iff 0\in \overline{\Tlocus{\mzeta}{\proto{\multi{}}{l}}}\,.$$
\end{pro}
\begin{proof}
	The germ $f:(M,f^{-1}(y))\to (N,y)$ is $\A$-equivalent to $\proto{\multi{}}{l}\times\id_{\CC^a}$. The stratification by singularity types is defined on the level of germs, cf. Remark \ref{rem:germs}. Thus, $f$ and $\proto{\multi{}}{l}\times\id_{\CC^a}$ induce the same germ of stratification at $y$ and
	$$y\in \overline{\Tlocus{\mzeta}{f}} \iff
	0 \in \overline{\Tlocus{\mzeta}{\proto{\multi{}}{l}\times\id_{\CC^a}}} \iff
	0 \in \overline{\Tlocus{\mzeta}{\proto{\multi{}}{l}}} \,.$$
\end{proof}
\begin{pro} \label{pro:codim}
    Let $\multi{}$ be a multisingularity that occurs for $l\in\NN$. For any stable map of relative dimension $l$ the stratum $\Tlocus{\multi{}}{f}$ is smooth. Its codimension is the dimension of the target space of the prototype $\proto{\multi{}}{l}$.
\end{pro}
\begin{proof}
    Choose a point $y \in \Tlocus{\multi{}}{f}$. Smoothness at $y$ may be checked in a small analytic neighborhood of $y$. There the germ of $f$ and $\proto{\multi{}}{l}\times\id_{\CC^a}$ coincide. The germ of the stratum $\Tlocus{\multi{}}{f}$ at $y$ is of the form $0\times\CC^a$, cf. Proposition \ref{pro:proto_0}.
\end{proof}

\subsection{Invariants of singularities}
\begin{df}
Let $\multi{}$ be the multisingularity corresponding to the algebra $Q$ and $l\in\NN$.
\begin{itemize}
    \item The dimension of $\multi{}$ is defined as the dimension of $Q$.
    \item The S-codimension and T-codimension ($\scodim_l(\multi{})$ and $\tcodim_l(\multi{})$) of  $\multi{}$ are defined to be the dimensions of the source and target spaces of its prototype for~$l$. 
\end{itemize}
\end{df}
\begin{rem} \label{rem:codim}
    Let $f:M\to N$ be a stable map of  relative dimension $l$. Proposition \ref{pro:codim} shows that $\tcodim_l(\multi{})=\codim(\Tlocus{\multi{}}{f}\subset N)$.
\end{rem}
For a singularity, the notion $\scodim_l$ coincides with its codimension in $\E(m,m+l)$. By definition, we have $\tcodim_l(\multi{})=\scodim_l(\multi{})+l$ for any $\multi{}$, and 
\[
\scodim_l(\eta_1,\ldots,\eta_r)=\sum_{i=1}^r \scodim_l(\eta_i) + (r-1)l,
\qquad
\tcodim_l(\eta_1,\ldots,\eta_r)=\sum_{i=1}^r \tcodim_l(\eta_i).
\]
The dimension $\dim \multi{}$ does not depend on $l$, while the S-codimension and T-codimension do. Construction of prototypes imply the following formula, cf. \cite[p.5]{primer}.
\begin{align}\label{w:scodim}
    \scodim_l(\multi{})&=(\dim(\multi{})-1)\cdot l+b(\eta)
    \,,&
    \tcodim_l(\multi{})&=\dim(\multi{})\cdot l+b(\eta)\,,
\end{align}
For a constant $b(\eta)$ independent of $l$. In Corollary \ref{cor:Der} we provide an explicit formula for this constant.
\subsection{Mather dimensions}
For $l\geq 1$ we define the {\em Mather bound}
    \begin{equation}\label{eq:Mather bound}
    M(l)=\begin{cases} 6l+8 & \text{if } l=1,2,3, \\ 6l+7 & \text{if } l\geq 4. \end{cases}
    \end{equation}
It is a fact (\cite{mather3, mather4}) that there are only finitely many stable singularities with codimension at most $M(l)$ in $\E(m,m+l)$, for any $m$.
Moreover, for any $k\le M(l)$ the subset of $\E(m,m+l)$ corresponding to singularities of codimension at least $k+1$ is of codimension at least $k+1$. For large $m$ the list of singularities with codimension at most $M(l)$
is the same.
\begin{df} 
\label{def:Mather_singularity}
    A (multi)singularity $\multi{}$ is called a Mather (multi)singularity for $l$ if 
    $$\scodim_l(\multi{})\le M(l)\,.$$
    An algebra $Q$ is called Mather algebra if the corresponding multisingularity is Mather.
\end{df}
For $l=1$ there are 32 Mather singularities and 265 Mather multisingularities. For $l\ge 18$ the list of Mather algebras does not depend on $l$. These are precisely all algebras of dimension at most $6$, together with a few of dimension 
$7$, cf.~Formula~\eqref{w:scodim}. \\
For the complete list of local Mather algebras for all $l$, see \cite{TPP}.

Mather singularities are special for several reasons. One of them is that their prototypes are stabilized by a positive 
$\CC^*$-action. In the case of singularities, this was proved by a case-by-case analysis; see \cite[Thm.~7.6]{MNB}. The generalization to multisingularities is obtained using Proposition~\ref{pro:proto_multi}.
\begin{thm} \label{thm:quasihomogenous}
    Let $\multi{}$ be a Mather multisingularity for $l$. There exists a positive $\CC^*$-action stabilizing the prototype $\proto{\multi{}}{l}$.
\end{thm}

\section{Thom polynomials} \label{s:Thom}
For this section fix a positive integer $l\ge 1$.
\subsection{Characteristic classes}
\begin{df}
    For a stable map $f:M\to N$ of relative dimension $l$  let $T_f=f^*TN-TM$ be its virtual normal bundle.
    \begin{enumerate}
        \item The Chern classes $c_\bullet(f):=c_\bullet(T_f)$ are defined by the formula 
        $$1+c_1(f)+\dots=\frac{1+f^*c_1(TN)+\dots+f^*c_n(TN)}{1+c_1(TM)+\dots +c_m(TM)} \in   \coh^\bullet(M)\,.$$
        For a partition $\lambda=(\lambda_1\geq \ldots \geq \lambda_r)$ define $c_\lambda(f)=\prod_{i=1}^kc_{\lambda_i}(T_f)$.
        \item The Landweber-Novikov classes $s_\lambda(f):=s_\lambda(T_f)$ are 
        $$s_{\lambda}(f):=f_*c_\lambda(f) \in \coh^\bullet(N).$$ 
    \end{enumerate}
\end{df}
Let $\cvar=(c_1,c_2,\dots)$ denote a set of variables indexed by natural numbers and $\svar=(s_\lambda)$ a set of variables indexed by partitions.  We consider gradations on the polynomial rings $\QQ[\cvar]$ and $\QQ[\svar]$ given by 
$$\deg(c_t)=t\,, \qquad \deg(s_\lambda)=l+\sum\lambda_i\,.$$
For a stable map $f$ and a polynomial $P\in\QQ[\cvar]$ we define the element $P(f)\in \coh^\bullet(M)$ by the substitution $c_i\to c_i(f)$. Analogously for $W\in\QQ[\svar]$ we define $W(f)\in \coh^\bullet(N)$ by the substitution $s_\lambda\to s_\lambda(f)$
\subsection{Thom polynomials}
One of the key tools of singularity theory are Thom polynomials. They are universal objects encoding characteristic classes of singularity loci for all stable maps.
We present only a brief introduction to the theory, for an overview, see the surveys \cite{primer,OhmotoSurvey}.
The foundational result of the theory is
\begin{thm}
    For every singularity $\eta$ there exists a polynomial $\Thom{\eta} \in \QQ[\cvar]$, called the Thom polynomial of $\eta$, such that for every stable map $f:M\to N$ of relative dimension $l$  we have
	$$\Thom{\multi{}}(f)=[\overline{\Qlocus{\eta}{f}}] \in \coh^\bullet(M)\,. $$
    The notation $[-]$ denotes the fundamental class.
\end{thm}
This result can be generalized to the case of multisingularities. It was established by Kazarian~\cite{KazaMulti} and proved by Ohmoto~\cite{TOmult}.
\begin{thm}[{\cite[Thm. 4.17]{TOmult}}]
    For every multisingularity $\multi{}$ there exist polynomials $\KazT{\multi{}}$ and $\ThT{\multi{}} \in \QQ[s_\bullet]$, such that:
    \begin{enumerate}
    \item
    For every stable map $f:M\to N$ of relative dimension $l$  we have
	$$\ThT{\multi{}}(f)=[\overline{\Tlocus{\multi{}}{f}}]\cdot|\Aut(\multi{})| \in \coh^\bullet(N)\,. $$
    \item
    The polynomial $\KazT{\multi{}}$ is linear. 
    \item 
    The polynomials $\KazT{\multi{}}$ determine $\ThT{\multi{}}$. Let $X$ be a finite set of singularities. Denote by $X_T$ the set of nonempty multi\-sin\-gu\-la\-ri\-ties, consisting of singularities from $X$. To every $\eta\in X$ we associate a formal variable $t_\eta$ and to every $\multi{}\in X_T$ a monomial $t^\multi{}$. Then
    \begin{align*} 
	1+\sum_{\multi{}\in X_T}
	\frac{\ThT{\multi{}}}{|\Aut(\multi{})|}\cdot t^\multi{}=
	\exp \bigg(
	\sum_{\multi{}\in X_T}
	\frac{\KazT{\multi{}}}{|\Aut(\multi{})|}\cdot t^\multi{}
	\bigg) \in \QQ[\svar][[\tvar]]
	\,.
    \end{align*}
    \end{enumerate}
\end{thm}

\subsection{Interpolation Method} \label{sec:interpolation}
The theory of global singularities is inherently computational: explicit Thom polynomials often translate into concrete results in enumerative and algebraic geometry, as well as obstruction theory. One of the most efficient methods for computing Thom polynomials is the interpolation method, introduced in \cite{rrtp}. It is based on the observation that Thom polynomials of Mather singularities are uniquely determined by their values on the specific set of test maps---prototypes of singularities. Now we recall a generalization of this method to the case of multisingularities.
\begin{thm}
    Suppose that $\multi{}$ is a Mather multisingularity. The Thom polynomial $\ThT{\multi{}} \in \QQ[\svar]$ is uniquely determined by the following properties.
    \begin{enumerate}
        \item $\ThT{\multi{}}$ is homogeneous of degree $\tcodim_l(\multi{})$.
        \item Let $\proto{\multi{}}{l}:\bigsqcup V_i\to W$ be the prototype of $\multi{}$ and $\TT_\multi{}$ the maximal torus of its stabilizer group. Let $0\in W$ be the origin. We have
        $$\ThT{\multi{}}(\proto{\multi{}}{l})=[0]\cdot |\Aut(\multi{})| \in \coh^{\bullet}_{\TT_\multi{}}(pt)\,.$$
        \item Let $\mzeta \neq \multi{}$ be a multisingularity of codimension at most $\tcodim_l(\multi{})$ and $\proto{\mzeta}{l}$ its prototype. Let $\TT_\mzeta$ be the maximal torus stabilizing $\proto{\mzeta}{l}$. We have
        $$\ThT{\multi{}}(\proto{\mzeta}{l})=0 \in \coh^{\bullet}_{\TT_\mzeta}(pt)\,.$$
    \end{enumerate}
\end{thm}

In conditions (2) and (3)—as well as elsewhere in this paper—an important point is that the Thom polynomial is evaluated on a map (namely the prototype $\theta$) in equivariant cohomology, rather than in ordinary cohomology. Indeed, the ordinary cohomology of both the source and the target of $\theta$ is trivial, whereas the equivariant cohomology rings $\coh^\bullet_{\TT_\multi{}}(pt)$ are isomorphic to polynomial rings. Consequently, substituting the Landweber–Novikov classes of the prototype $\proto{\mzeta}{l}$ into such a polynomial is governed by the weights of the $\TT_\mzeta$-action on the source and target of $\proto{\mzeta}{l}$. These weights are listed on \cite{TPP}. 

In effect, conditions (2) and (3) are linear equations in the coefficients of the Thom polynomial $\ThT{\multi{}}$. Therefore, computing the Thom polynomial of a Mather multisingularity reduces to solving a system of linear equations; which can be efficiently handled by a computer algebra system. \\
There is a version of the interpolation method for the polynomials $\KazT{\multi{}}$ instead of $\ThT{\multi{}}$.
It is usually more convenient for computations, since the polynomials $\KazT{\multi{}}$ are linear. Moreover, it suffices to consider prototypes of singularities as test maps, rather than prototypes of multisingularities.
\begin{thm} \label{thm:interp2}
    Let $\mzeta$ be a Mather multisingularity and $X$ the set of singularities appearing in $\mzeta$. Denote by $X_T$ the set of nonempty multi\-sin\-gu\-la\-ri\-ties, consisting of singularities from $X$. Suppose that to every multisingularity $\multi{}\subset\mzeta$ we associated a linear polynomial $\KazT{\multi{}}$ and that the polynomials $A_\multi{}$ are defined by equation:
      \[
    1+\sum_{\multi{}\in X_T}
	\frac{A_{\multi{}}}{|\Aut(\multi{})|}\cdot t^\mzeta=
	\exp \bigg(
	\sum_{\multi{}\subset \mzeta}
	\frac{\KazT{\multi{}}}{|\Aut(\multi{})|}\cdot t^\multi{}
	\bigg) \in \QQ[\svar,\tvar]
	\,.
    \]
    Assume that:
    \begin{enumerate}
        \item The polynomial $\KazT{\multi{}}$ is homogeneous of degree $\tcodim_l(\multi{})$.
        \item Let $\eta\in X$ be a singularity, $\proto{\eta}{l}$ its prototype and $\TT_\eta$ the maximal torus of the stabilizer group. We have
        $$ \KazT{\eta} (\proto{\eta}{l})=[0] \in \coh^{\bullet}_{\TT_\eta}(pt)\,.$$
        \item Let $\multi{} \subset\mzeta$ and $\eta$ be a singularity (not necessarily in $X$) of codimension at most $\tcodim_l(\multi{})$, such that $\multi{}\neq \eta$. Let $\proto{\eta}{l}$ denote the prototype of $\eta$ and $\TT_\eta$ be the maximal torus of its stabilizer group. We have
        $$A_\multi{}(\proto{\eta}{l})=0 \in \coh^{\bullet}_{\TT_\eta}(pt)\,.$$
    \end{enumerate}
    Then $A_\multi{}=\ThT{\multi{}}$ for every $\multi{}\subset\mzeta$.
\end{thm}

This theorem was known, or at least widely believed, by experts in the field. In concrete applications a formal proof is not required. In such situations, it suffices to list the requirements (1), (2), and (3) and observe that they admit a unique solution; see, for example, \cite{rrmult, KazaMulti, RM4tuple, ohmotoSMTP, PallaresPenafort}. 
A generalization of this theorem is available in \cite[Thm. 8.2]{JKRRmult}. It involves a one-parameter deformation, called the Segre–Schwartz–MacPherson–Thom polynomials.

\begin{rem}
    As seen in the theorems above, symmetries of singularities play a key role in calculating Thom polynomials. The necessary data on the symmetries of all Mather singularities are given on \cite{TPP}, see also Section~\ref{sec:symmetries_summary} below. 
\end{rem}

\section{Symmetries as consistency check for algebra classification}

In this section, we discuss symmetry groups. They are used in two different ways in this paper. First, we saw in Section~\ref{sec:interpolation} that the interpolation method explicitly depends on symmetries of singularities. Now we show a second argument, namely how symmetries can be {used} to {\em test} whether a conjectured list of small dimensional algebras is complete or not (cf. Section~\ref{sec:algebras}).  

\subsection{Symmetries of singularities} \label{sec:symmetries_summary}
We begin by briefly summarizing results from \cite[Ch.1]{rrphd} on symmetries of Mather singularities and their associated algebras. This section relies more heavily on notions and standard facts from singularity theory than the rest of the paper; the monograph \cite{MNB} is a standard reference.

Let $l\geq 0$, and let $Q$ be a Mather algebra for $l$, minimally presented by $a$ generators and $b$ relations: $Q=\CC[x_1,\ldots,x_a]/(r_1,\ldots,r_b)$. The difference $l_0:=b-a$ is called the {\em defect} of the algebra, and we necessarily have $l\geq l_0$. We will be concerned with some groups (up to homotopy equivalence) and some germs. 
\begin{itemize}
    \item Let $G_Q$ denote the maximal compact subgroup of the algebra automorphism group of $Q$.
       \item The {\em $l$-genotype} associated to $Q$ is the germ
\[
\kappa_Q^l:\CC^a\to \CC^{a+l},
\qquad
(x_1,\ldots,x_a)\mapsto (r_1,\ldots,r_b,\underbrace{0,\ldots,0}_{l-l_0}).
\]
\item Let $m\geq a$. The {\em contact group} $\K=\K(m,m+l)$ acts on $\E(m,m+l)$, and for the orbit of the (trivial unfolding of the) genotype we have
\[ 
\K \cdot (\kappa_Q^l \times \id_{\CC^{m-a}})= \eta_{Q} \subset \E(m,m+l)
\]
from \eqref{eq:eta_Q}. 
 \item The so-called {\em miniversal unfolding} of the $l$-genotype is the $l$-prototype 
$
\proto{Q}{l}: \CC^{m_0} \to \CC^{m_0+l}
$.
\item Let $m \geq m_0$. The {\em right-left group} $\A=\A(m,m+l)$ acts on $\E(m,m+l)$, and define
\[
\A \cdot (\proto{Q}{l} \times \id_{\CC^{m-m_0}}) = \eta'_Q.
\]
\item We have $\eta'_Q \subset \eta_Q$, and the difference $\eta_Q -\eta'_Q$ has smaller dimension than $\eta_Q$. In fact, $\eta'_Q$ is the unique $\A$-orbit in $\eta_Q$ consisting of stable germs. We have
\[
\codim(\eta_Q \subset \E(m,m+l))=\codim(\eta'_Q \subset \E(m,m+l))=m_0,
\]
the source dimension of the $l$-prototype. 
\item 
The $\K$-stabilizer subgroup of the genotype $\kappa_Q^l$, up to homotopy equivalence, is $G_Q \times U(l-l_0)$.  The $\K$-stabilizer subgroup of an element of the $\K$-orbit $\eta_Q$, up to homotopy equivalence, is $G_Q \times U(l-l_0) \times U(m-a)$.
(Even though 
$\K$ is infinite dimensional, it is possible to define the maximal compact subgroup of the stabilizer, essentially by considering high enough jets instead of germs, see details in \cite[Ch.1]{rrphd}. This maximal compact symmetry group turns out to be homotopy equivalent, in the appropriate sense, to the named algebraic groups.)
\end{itemize}

\begin{ex}
    For $Q=\III_{23}=\CC[x,y]/(x^2,xy,y^3)$ we have $l_0=1$, and $G_Q=U(1)\times U(1)$ with the components scaling $x$ and $y$. 
    The genotype $\kappa^2$ for $l=2$ is the germ
    \[
    \kappa^2:(\CC^{2},0)\to (\CC^{4},0), \qquad
    (x,y)\mapsto (x^2,xy,y^3,0).
    \]
    The prototype $\proto{}{2}$ for $l=2$ is the miniversal unfolding germ $(\CC^{2+9},0)\to (\CC^{4+9},0)$,
\begin{multline*}
\proto{}{2}: (x,y,u_1,\ldots,u_9) \mapsto \\
(x^2+u_1x+u_2y+u_3y^2, xy, y^3+u_4x+u_5y+u_6y^2,u_7x+u_8y+u_9y^2,
    u_1,\ldots, u_9).
\end{multline*}
The maximal compact $\K$-symmetry group of $\kappa^2$ (as well as the maximal compact $\A$-symmetry group of $\proto{}{2}$) is $G_Q \times U(1)$. The reader may find it instructive to find the (in this case diagonal) actions of this group on the two source and target spaces. Hence, for $m\geq 2$ the maximal compact $\K$-stabilizer of a representative in $\eta_{\III_{23}}\subset \E(m,m+2)$ is
\[
\left( G_Q \times U(1) \right) \times U(m-2) = U(1)^3 \times U(m-2).
\]
The codimension of $\eta_{\III_{23}}$ in $\E(m,m+2)$ is 11, the source dimension of $\proto{}{2}$. 
\end{ex}

\subsection{Poincar\'e series identities for Mather singularities}
\label{sec:classification_vs_identity}

Let $l\geq 0$ be fixed, and let $\eta_1,\eta_2,\ldots,\eta_N$ be the list of Mather singularities that occur for the relative dimension~$l$. The corresponding local algebras are denoted by $Q_1,Q_2,\ldots,Q_N$, the corresponding $l_0$ values (defects) are denoted $l_{0i}$, and the corresponding $m_{0}$ values are denoted $m_{0i}$. 
Recall that $m_{0i}$ be the codimension of $\eta_i\subset \E(m,m+l)$ for large enough $m$. 

For an algebraic group $G$ define
\[
\PP_G = \sum_{j=0}^{\infty} \dim \coh^{2j}(BG;\QQ)q^j.
\]

Our main theorem in this section follows from a version of 
Kazarian's spectral sequence; the spectral sequence of a filtration applied to the Borel construction on the representation. Classical references include \cite{AB} (implicitly), \cite{KazaMulti}, and more recent related expositions and examples are in  \cite[\S10]{LFRRobstructions}, \cite[\S1]{rrCOHA}, \cite[\S5]{RRSimon}.

\begin{thm} \label{thm:MatherSS}
Up to (and including) cohomological degree $2M(l)$ we have
\begin{equation}
\label{eq:SSidentity}
\PP_{U(\infty)}=
\sum_{i=1}^N 
q^{m_{0i}} \PP_{G_{Q_i}}\PP_{U(l-l_{0i})},
\end{equation}
that is, the difference of the two sides is $o(q^{M(l)})$.
\end{thm}

Observe that the left hand side is independent of $l$: it is the generating function of the number of partitions.  

\begin{ex} \label{ex:ss_example}
    For $l=0$ we have $N=20$. The first few algebras are $A_0$, $A_1$, $A_2$, $A_3$, $A_4$, $I_{22}$, $A_5$, $I_{23}$, of codimensions 0,1,2,3,4,4,5,5, respectively. Their defects are all $0$. We have $G_{A_0}=1$, $G_{I_{22}}=(U(1)\times U(1))\rtimes S_2$, and the other symmetry groups are $U(1)$. Figure~\ref{fig:l=0 ss} illustrates the nature of the identity \eqref{eq:SSidentity}.
\begin{figure}
 \begin{center}
    \begin{tabular}{r|rrrrrrrr} 
    
    $A_0$ & 1 &    &  &  &  &  & & \\
    $A_1$ &   & $q$  & $q^2$   & $q^3$ & $q^4$ & $q^5$ &  $q^6$ & $\ldots$\\
    $A_2$ &   &     & $q^2$     & $q^3$   & $q^4$ & $q^5$ &  $q^6$ & $\ldots$ \\
    $A_3$ &   &     &       & $q^3$     & $q^4$   & $q^5$ & $q^6$ &$\ldots$\\
    $A_4$ &   &     &       &       & $q^4$     & $q^5$   & $q^6$ &$\ldots$\\
    $A_5$ &   &     &       &       &       & $q^5$     & $q^6$ &$\ldots$\\
 $I_{22}$ &   &     &       &       &  $q^4$&  $q^5$  & $2q^6$ &$\ldots$\\
$I_{23}$ &   &     &       &       &  &  $q^5$  & $q^6$ &$\ldots$\\
$\vdots$   &  &     &       &       &  &   & $\ddots$\\
  \hline
  generating function  \\ 
 of partitions  & 1 & $q$ & $2q^2$ & $3q^3$ & $5q^4$ & $7q^5$ & $11q^6$ &$\ldots$\\ 
    \end{tabular}
\end{center}
    \caption{Illustration of Theorem~\ref{thm:MatherSS} for $l=0$. The rows show the contribution of Mather singularities to the RHS. Hence, the rows must add up to the generating function of partitions. Since all singularities are listed of codim$\leq 5$, the displayed part of the table (correctly) `predicts' that there must be $11-8=3$ singularities (or corresponding algebras) of codimension~6.}
    \label{fig:l=0 ss}
\end{figure}
\end{ex}

\begin{proof}
Let $m\gg 0$, and define $\E'=\cup_{i=1}^N \eta_i \subset \E(m,m+l)$. According to Mather theory the complement $\E(m,m+l) -\E'$ has codimension greater than $M(l)$ in $\E(m,m+l)$. The group $\K(m,m+l)$ is homotopy equivalent to $\A(m,m+l)$, and further to $GL_m(\CC)\times GL_{m+l}(\CC)$. Hence, we have that up to (and including) cohomological degree $2M(l)$ we have
\begin{multline} \label{eq:something}
    \coh^\bullet_{\K}(\E')=\coh^\bullet_{\K}(\E(m,m+l))= \coh^\bullet_{GL_m(\CC)\times GL_{m+l}(\CC)}(\E(m,m+l)) \\= \coh^\bullet(B(GL_m(\CC)\times GL_{m+l}(\CC))).
\end{multline}
Now we will calculate $\coh^\bullet_{\K}(\E')$ using data of $\E'$'s parts $\eta_i$. The numbers $l_0$ and $m_0$ associated with $\eta_i$ will be called $l_{0i}$ and $m_{0i}$. Consider 
\[
{\mathfrak F}_p= \bigcup_{\codim(\eta_i ) \leq p} \eta_i,
\]
and the spectral sequence $E^*_{**}$ in $\K$-equivariant cohomology associated with the filtration 
\[
\emptyset \subset {\mathfrak F}_0 \subset {\mathfrak F}_1 \subset \ldots 
=\E'.
\]
The spectral sequence  converges to $\coh^\bullet_{\K}(\E')$.
Its $E^2$ page is given by
\[
E^2_{2p,*}=\bigoplus_{\codim(\eta_i) = p} \coh^\bullet(BG_i),
\]
where $G_i$ is the $\K$-stabilizer subgroup (or its homotopically equivalent maximal compact subgroup) of a representative in $\eta_i$. For those groups the classifying spaces have no odd cohomology, and hence the spectral sequence is supported on $E^*_{even,even}$, and therefore it degenerates at $E^2$. Then, using \eqref{eq:something} and the statements in Section~\ref{sec:symmetries_summary}, up to cohomological degree $2M(l)$, we obtain
\[
\PP_{GL_m(\CC) \times GL_{m+l}(\CC)}
=
\sum_{i=1}^N
q^{m_{0i}} \PP_{G_i}
=
\sum_{i=1}^N
q^{m_{0i}} \PP_{G_{Q_i}} \PP_{U(l-l_{0i})} \PP_{U(m-a)}.
\]
As $m\to \infty$ one $\PP_{GL_{\infty}(\CC)}=\PP_{U(\infty)}$ factor cancels and we obtain, up to cohomological degree $2M(l)$, that 
\[
\PP_{GL_\infty(\CC)}
=
\sum_{i=1}^N
q^{m_{0i}} \PP_{G_{Q_i}} \PP_{U(l-l_{0i})},
\]
what we wanted to prove.   
\end{proof}

The role of Theorem~\ref{thm:MatherSS} in the classification of algebras/singularities is clear: it serves as a `consistency check'. Any error in the classification of small-dimensional algebras would likely violate identity~\eqref{eq:SSidentity}. For example, omitting an algebra with corresponding singularity $\eta$ causes the identity to fail at the coefficient of $q^c$, where $c=\codim(\eta)$ (cf.~Figure~\ref{fig:l=0 ss}). Moreover, \eqref{eq:SSidentity} provides not a single such test, but infinitely many—one for each~$l$.

The website \cite{TPP} contains the list of Mather singularities $Q_i$, together with data on the groups $G_{Q_i}$. That data is consistent with Theorem~\ref{thm:MatherSS}. We verified this fact by computer for $l=0,1,\ldots,100$, and with an extra algebraic argument (not detailed here) for $l>100$.

\section{Stable hierarchy}
    \subsection{Definition and structure of the partial order}
    Fix an integer $l\ge 1$.
    Let $Q$ be a local algebra that occurs for $l$ and $\eta$ the corresponding singularity. 
    The classical partial order
    on singularities is induced by containment in the vector space of germs, see e.g. \cite[Sect. 15]{Varchenko_book}, \cite{rrtp,FeherPatakfalvi}. We refer to this order as the stable hierarchy for $l$ and denote it by $\le_l$.
    \begin{df}\label{df:hier_sing}
        Let $Q$ and $R$ be local algebra that occurs for $l$. Let $\eta$ and $\zeta$ be the corresponding singularities. We write $Q\le_l R$ (or $\eta \le_l \zeta$) if one of the following equivalent conditions hold:
        \begin{enumerate}
            \item The singularity $\eta$ is contained in the closure of $\zeta$ in the germ space $\E(m,m+l)$ for large enough $m$.
            \item For every stable map of relative dimension $l$  we have $\Qlocus{Q}{f}\subset \overline{\Qlocus{R}{f}}$.
        \end{enumerate}
    \end{df}
    We
    generalize this partial order to the case of multisingularities. Here, there is no reasonable substitute for the ambient space $\E(m,m+l)$. The standard tool of singularity theory, the multijet bundle, is not well suited to our purposes, since it fixes the number of components
    of a multisingularity. Thus, it does not allow e.g. a comparison between $A_0^2$ and $A_1$. Instead, we study the singularity loci for stable maps. 
    \begin{df} \label{pro:defT}
	Let $\multi{}$ and $\mzeta$ be multisingularities that occur for $l$.  We write $\multi{} \le_l \mzeta$ if one of the following equivalent conditions hold:
	\begin{enumerate}
		\item For any stable map $f$ of relative dimension $l$  we have
		$\Tlocus{\multi{}}{f} \subset \overline{\Tlocus{\mzeta}{f}}\,.$
		\item There exists a stable map $f$ of relative dimension $l$  such that
		$\Tlocus{\multi{}}{f} \cap \overline{\Tlocus{\mzeta}{f}}\neq \varnothing\,.$
		\item Let $\proto{\multi{}}{l}$ be the prototype of $\multi{}$ for $l$. We have $0\in \overline{\Tlocus{\mzeta}{\proto{\multi{}}{l}}}$.
	\end{enumerate}
    \end{df}
    The equivalence of conditions (1), (2), and (3) in the above definition is an immediate consequence of Proposition \ref{pro:local}.
\begin{rem}
    The equivalence of (1) and (2) in Definition \ref{pro:defT} means that for any stable map the stratification by singularity types satisfies the frontier condition.
\end{rem}
We use the same symbol $\le_l$ to denote the stable hierarchy of singularities and multisingularities. In Proposition \ref{pro:two_orders} we will show that this notation is consistent: the order for singularities is simply the restriction of the order on multisingularities. Before proving this, we show that the stable hierarchy $\le_l$ on multisingularities is determined by combinatorics and ``elementary splits'' of the form $\eta\le_l\mzeta$ where $\eta$ is a {\em mono}singularity. 
\begin{pro} \label{pro:splits}
	Let $\multi{}=(\eta_1,\dots,\eta_k)$ and $\mzeta$ be multisingularities that occur for $l$. We have $\multi{} \le_l \mzeta$ if and only if there exist multisingularities $\mzeta_1,\dots,\mzeta_{k}$ such that:
    \begin{enumerate}
        \item The multisingularities $\mzeta_i$ sum up to $\mzeta$, i.e. $\mzeta=\mzeta_1+\dots+\mzeta_{k}$.
        \item For any index $i$ we have $\eta_i \le_l \mzeta_i$\,.
    \end{enumerate}
    \end{pro}
    \begin{proof}
    Let $\proto{\eta_i}{l}:V_i\to W_i$ be the prototype of $\eta_i$ and $W=\prod_{i=1}^k W_i$. Proposition \ref{pro:proto_multi} implies that the prototype of the multisingularity $\multi{}$ is of the form
	$$\proto{\multi{}}{l} 
    :\bigsqcup_{i=1}^k \big(V_i\times \prod_{j\neq i} W_j\big) \to  W\,.$$
    The stratification of $W$ by singularity types induced by $\proto{\multi{}}{l}$ is the intersection of stratifications induced by $\proto{\eta_i}{l} \times \id$. In particular
    \begin{align} \label{1}
		\Tlocus{\mzeta}{\proto{\multi{}}{l}}=
        \bigcup\limits_{\mzeta=\mzeta_1+\dots+\mzeta_{k}}\bigcap_{i=1}^k\Tlocus{\mzeta_i}{\proto{\eta_i}{l}\times \id}
        =
        \bigcup\limits_{\mzeta=\mzeta_1+\dots+\mzeta_{k}}\prod_{i=1}^k \Tlocus{\mzeta_i}{\proto{\eta_i}{l}}
        \,.
	\end{align}   
    We have $\multi{} \le_l \mzeta$ if and only if the origin $0\in W$ lies in the closure of the stratum $\Tlocus{\mzeta}{\proto{\multi{}}{l}}$. The sum in Formula \eqref{1} is finite, so
    the closure of the union equals the union of the closures, thus
    $$
    \multi{} \le_l \mzeta \iff
    0 \in\bigcup\limits_{\mzeta=\mzeta_1+\dots+\mzeta_{k}}\prod_{i=1}^k \overline{\Tlocus{\mzeta_i}{\proto{\eta_i}{l}}}\,.
    $$
    This means exactly that there is a decomposition $\mzeta=\mzeta_1+\dots+\mzeta_{k}$ such that for every $i$ the origin $0\in W_i$ lies in the closure of
    $\Tlocus{\mzeta_i}{\proto{\eta_i}{l}}$.
\end{proof}
\begin{cor} \label{cor:dense_sing}
    Let $f$ be a stable map of relative dimension $l$  and $\eta$ a singularity that occurs for $l$. The locus $\Tlocus{\eta}{f}$ is dense in $f(\Qlocus{\eta}{f})$.
\end{cor}
\begin{proof}
    We have $l\ge 1$, therefore the image of the map $f$ is of codimension at least $1$. This implies that the empty multisingularity is the maximal element of the order $\le_l$. Thus, Proposition \ref{pro:splits} implies that
    $
    \eta\in\multi{} \Rightarrow \multi{}\le_l\eta\,.
    $
    This means that for every $\eta\in\multi{}$ we have
    $$\Tlocus{\multi{}}{f}\subset \overline{\Tlocus{\eta}{f}}$$
    By Proposition \ref{pro:loci} the loci $\Tlocus{\multi{}}{f}$ sum up to $f(\Qlocus{\eta}{f})$.
\end{proof}
We are now ready to show that the restriction of the stable hierarchy from Definition \ref{pro:defT} to singularities is the partial order from Definition \ref{df:hier_sing}.
\begin{pro} \label{pro:two_orders}
    Let $\eta$ and $\zeta$ be singularities that occur for $l$ and $f$ a stable map of relative dimension $l$ . Suppose that  $\Tlocus{\eta}{f}\neq \varnothing$, then
    $$
    \Qlocus{\eta}{f}\subset \overline{\Qlocus{\zeta}{f}} \iff \Tlocus{\eta}{f}\subset \overline{\Tlocus{\zeta}{f}}\,.
    $$
\end{pro}
\begin{proof}
    Through the proof we will repeatedly use Corollary \ref{cor:dense_sing} and the fact that the map $f$ is closed, cf. Remark \ref{rem:closed}. \\
    Suppose that $\Qlocus{\eta}{f}\subset\overline{\Qlocus{\zeta}{f}}$.  Then
    $$
     \Tlocus{\eta}{f}\subset f(\Qlocus{\eta}{f}) \subset
     f\Big(\overline{\Qlocus{\zeta}{f}}\Big)=
     \overline{f(\Qlocus{\zeta}{f})}=
     \overline{\Tlocus{\zeta}{f}}
     \,.
    $$
    Suppose now that $\Tlocus{\eta}{f}\subset \overline{\Tlocus{\zeta}{f}}$. We have
    $$\Tlocus{\eta}{f}\subset\overline{\Tlocus{\zeta}{f}}=\overline{f(\Qlocus{\zeta}{f})}= f\Big(\overline{\Qlocus{\zeta}{f}}\Big)\,.$$
    Every point in $\Tlocus{\eta}{f}$ has a unique preimage, thus
    $f^{-1}(\Tlocus{\eta}{f})\subset \overline{\Qlocus{\zeta}{f}}$. On the other hand
    $$
    \Qlocus{\eta}{f}\subset
    \overline{f^{-1}(f(\Qlocus{\eta}{f}))}=
    f^{-1}\Big(\overline{f(\Qlocus{\eta}{f})}\Big)=
    f^{-1}\Big(\overline{\Tlocus{\eta}{f}}\Big)=
    \overline{f^{-1}(\Tlocus{\eta}{f})}\,.
    $$
\end{proof}
The dimension of a multisingularity, and the codimension of singularity loci are consistent with the partial order $\le_l$. This generalizes standard invariants from the singularity case.
    \begin{pro}
        Let $\multi{}$ and $\mzeta$ be multisingularities that occur for $l$. If  $\multi{} \le_l \mzeta$ then
        \begin{itemize}
            \item $\tcodim_l \multi{} \ge \tcodim_l \mzeta$\,,
            \item $\dim \multi{} \ge \dim \mzeta$\,.
        \end{itemize}
    \end{pro}
    \begin{proof}
        Let $f:M\to N$ be a stable map of relative dimension $l$  such that $\Tlocus{\multi{}}{f}\neq \varnothing$. Then $\Tlocus{\multi{}}{f}\subseteq\overline{\Tlocus{\mzeta}{f}}$, thus 
        $\dim \Tlocus{\multi{}}{f}\le \dim \Tlocus{\mzeta}{f}\,.$ This implies the first part of the proposition, cf. Remark \ref{rem:codim}. \\
        For the second part, consider the sheaf $\mathcal{F}=f_*\O_M$. For any $y\in N$ denote its fiber at $y$ by $\mathcal{F}(y)$.  We have
        $\dim\multi{y}=\dim_\CC \mathcal{F}(y)\,$
        cf. Corollary \ref{cor:formal_local} and \cite[Ex. 14.3.M]{Vakil}. The claim follows from the semicontinuity of the rank 
        of a finite type quasicoherent sheaf \cite[Ex. 14.3.J]{Vakil}.
    \end{proof}
\subsection{Dependence on \texorpdfstring{$l$}{l}}
In this section, we investigate the dependence of the stable hierarchy $\le_l$ on  the natural number $l$. The first thing to note is that the set of stable multisingularities depends on $l$. 
Next, suppose that $\multi{}$ and $\mzeta$ are multisingularities that occur for natural numbers $l_1$ and $l_2$. It is natural to ask whether the relation $\multi{}\le_{l_1}\mzeta$ is equivalent to $\multi{}\le_{l_2}\mzeta$.
We provide a partial answer.
\begin{pro}\label{pro:infty}
	  Let $\multi{}$ and $\mzeta$ be multisingularities that occur for $l\in\NN$. Suppose that $\multi{}\le_l\mzeta$. Then $\multi{}\le_{l+1}\mzeta$.
\end{pro}
We will prove this fact later in this section. It motivates the following definition of the limit version of the stable hierarchy.
\begin{df}
    Let $\multi{}$ and $\mzeta$ be multisingularities. We write $\multi{}\le_\infty\mzeta$ if there exists a positive integer $l$ such that $\multi{}\le_l\mzeta$.
\end{df} 
We do not know whether the relation $\le_l$ can depend on $l$ in general. In Corollary \ref{cor:order} we prove that under the assumption $\dim\multi{}=\dim\mzeta$ it is independent of $l$. Moreover, if $\multi{}$ occurs for some $l$ and $\multi{}\le_\infty\mzeta$ then $\mzeta$ also occurs for this $l$. The example below shows that the second claim is not true in general.
\begin{ex} \label{ex:C3C41}
    Let us recall notation for algebras, cf. Section~\ref{sec:algebras}:
    $$
    C_{41}=\CC[x,y,z]/(xy,xz,yz,x^2-y^2,y^2-z^2)\,,\qquad
    C_3=\CC[x,y,z]/(x,y,z)^2
    $$
    We have $C_{41}\le_3 C_3$ and $C_{41}$ occurs for $l=2$. Yet, $C_3$ does not occur for $l=2$.
\end{ex}
\begin{con}
    Let $\multi{}$ and $\mzeta$ be multisingularities that occur for $l\in \NN$. Suppose that $\multi{}\le_\infty\mzeta$. We conjecture that $\multi{}\le_l\mzeta$.
\end{con}
Now we focus on the proof of Proposition \ref{pro:infty}. We need some notation. \\
Consider a stable germ $f:(\CC^m,0)\to (\CC^n,0)$ and pick a representing map
\begin{align} \label{w:unfolding0}
    f(\underline{x})=(f_1(\underline{x}),\dots, f_n(\underline{x}))\,.
\end{align}
Let $Q$ be its local algebra at $0$ and $\m$ the maximal ideal of $Q$. Let $h_1,\dots, h_k$ be monomials, whose classes in $Q$ form a basis of the ideal $\m$ treated as a vector space. The germ $(f_1,\dots, f_n,0):\CC^m\to \CC^{n+1}$ is no longer stable.
Its stable unfolding $F:\CC^{m+k}\to \CC^{n+1+k}$ is of the form:
\begin{align} \label{w:unfolding}
F(\xvar,\tvarr)=\left(f_1(\xvar),\dots, f_n(\xvar), t_1h_1(\xvar)+\dots+t_kh_k(\xvar),\tvarr\right)\,.
\end{align}
\begin{rem}
    The above construction is used to obtain prototypes of singularities. Suppose that $f$ is the prototype of a singularity $\eta$ for $l$. Then $F$ is the prototype of $\eta$ for $l+1$, cf. \cite[Example 2.9]{JKRRmult}.
\end{rem}
For  points $z\in\CC^m$ and $w\in\CC^{m+k}$, let $Q^f_z$ and $Q^F_w$ denote the local algebras of the maps $f$ and $F$ at the points $z$ and $w$.
We need the following technical lemma.
\begin{lemma} \label{lem:proto2}
	Consider points $\rvar=(r_1,\dots,r_k)\in \CC^k$ and $\avar=(a_1,\dots,a_m)\in \CC^m$. We have
	$$
	Q^F_{\avar,\rvar} \simeq \frac{Q^f_\avar}
	{(s)}\,,
	$$
	where $s=\sum r_i(h_i(\xvar+\avar)-h_i(\avar))$.
\end{lemma}
\begin{proof}
	Let $g_i(\xvar)=f_i(\avar+\xvar)-f_i(\avar)$.
We have
	$$
	Q^f_{\avar}\simeq
	\frac{\CC[[\xvar]]}{(g_1,\dots,g_n)}
	\,.
	$$
	In local coordinates near the point $(\avar,\rvar)$ the map $F$ is of the form
	\begin{align*}
	G(\xvar,\tvarr)=&\, F(\avar+\xvar,\rvar+\tvarr)-F(\avar,\rvar)
	\\=&
	\big(g_1(\xvar),\dots, g_n(\xvar),\sum (t_i+r_i)h_i(\xvar+\avar)-\sum r_ih_i(\avar),\tvarr\big)\,.
	\end{align*}
	Therefore, its local algebra is
	\begin{align*}
		Q^F_{\avar,\rvar}
		&\simeq\frac{\CC[[\xvar,\tvarr]]}{\big(g_1,\dots,g_n,\sum (t_i+r_i)h_i(\xvar+\avar)-\sum r_ih_i(\avar),\tvarr\big)} \\
		&\simeq\frac{\CC[[\xvar,\tvarr]]}{\big(g_1,\dots,g_n,\sum r_i(h_i(\xvar+\avar)-h_i(\avar)),\tvarr\big)} \\
		&\simeq\frac{\CC[[\xvar]]}{\big(g_1,\dots,g_n,\sum r_i(h_i(\xvar+\avar)-h_i(\avar))\big)}  \\
		&\simeq\frac{Q^f_\avar}{\sum r_i(h_i(\xvar+\avar)-h_i(\avar))}
		\,.
	\end{align*}
\end{proof}
\begin{cor} \label{cor:proto1}
	Consider a subspace $\CC^n\times0 \subset \CC^{n+k+1}$. For any  multisingularity $\multi{}$ we have
	$$
	\Tlocus{\multi{}}{F}\cap (\CC^n\times 0 )= \Tlocus{\multi{}}{f} \,.
	$$
\end{cor}
We are now ready to prove Proposition \ref{pro:infty}.
\begin{proof}[Proof of Proposition  \ref{pro:infty}]
    Due to Proposition \ref{pro:splits} it is enough to consider the case when $\multi{}$ is a {\em mono}singularity. Its prototype $\proto{\multi{}}{l}$ is a map between affine spaces, such as $f$ in \eqref{w:unfolding0}. Let $F$ be a stable unfolding constructed by the formula \eqref{w:unfolding} for $f=\proto{\multi{}}{l}$. \\
    By Definition \ref{pro:defT} (3) we have $0\in \overline{\Tlocus{\mzeta}{\proto{\multi{}}{l}}}$. Corollary \ref{cor:proto1} implies that $0\in \overline{\Tlocus{\mzeta}{F}}$. The map $F$ is stable of relative dimension $l+1$, thus by Definition \ref{pro:defT} (2) we have $\multi{}\le_{l+1}\mzeta$.
\end{proof}
As another implication of Lemma \ref{lem:proto2} we present a result about the hierarchy between algebras of different dimensions.
\begin{pro} \label{pro:quot}
	Let $(Q,\m)$ be a local algebra that occurs for $l$ and $s\in \m$. We have $Q \le_{l+1} Q/(s)$.
\end{pro}
\begin{proof}
    Let $\eta$ be the singularity corresponding to the algebra $Q$, and $\proto{\eta}{l}:\CC^m\to \CC^n$ its prototype. Let $F$ be a stable unfolding constructed by the formula \eqref{w:unfolding} for $f=\proto{\eta}{l}$. \\
    Elements $h_i$ form a basis of the ideal $\m$, thus there exist $r_1,\dots,r_k\in \CC^k$ such that
	$$r_1h_1+\dots+r_kh_k=s\,.$$
	Let $0_{n+1},0_m$ denote the origins in the vector spaces $\CC^{n+1}$ and $\CC^{m}$ respectively. Consider the sequence
    $$y_i=\frac{1}{i}\cdot(0_{n+1}, r_1,\dots,r_k) \in \CC^{n+k+1}\,.$$
    The point $\frac{1}{i}\cdot(0_{m}, r_1,\dots,r_k)$ lies in the preimage $F^{-1}(y_i)$. By Lemma \ref{lem:proto2} and Corollary \ref{cor:dense_sing} we have $y_i\in F(\Qlocus{Q/(s)}{F})\subset \overline{\Tlocus{Q/(s)}{F}}$.
    Thus,
    $$
    0\in\overline{\{y_i\}_{i\in\NN}} \subset \overline{\Qlocus{Q/(s)}{F}}\,.
    $$
    The claim follows from Definition \ref{pro:defT} (2).
\end{proof}
\begin{cor}
    Let $Q$ and $R$ be finite dimensional local algebras. Suppose that there exists a surjective morphism $\psi:Q\to R$. Then $Q\le_\infty R$.
\end{cor}
\subsection{Source version}
Up to now we considered the target version of the stable hierarchy. It describes the relative position of the singularity loci $\Tlocus{\multi{}}{f}$ in the target of a stable map $f$. In this section, we define a corresponding notion in the source, describing the relative position of the loci $\Slocus{\eta,\multi{}}{f}$.  We show that the obtained partial order is fully determined by the hierarchy in the target and combinatorics, yet not in an obvious way, cf.
Example \ref{ex:SandT}.
We omit most of the proofs, as they are analogous to their target counterparts.
\begin{df}
		Let $(\eta,\multi{})$ and $(\zeta,\mzeta)$ be S-multisingularities that occur for an $l\ge 1$. We write $(\eta,\multi{}) \le_l (\zeta,\mzeta)$ if the following conditions are satisfied
	\begin{enumerate}
		\item For any stable map $f$ of relative dimension $l$  we have
		$\Slocus{\eta,\multi{}}{f} \subset \overline{\Slocus{\zeta,\mzeta}{f}}\,.$
		\item There exists a stable map $f$ of relative dimension $l$  such that
		$\Slocus{\eta,\multi{}}{f} \cap \overline{\Slocus{\zeta,\mzeta}{f}}\neq \varnothing\,.$
		\item Let $\proto{\multi{}}{l}$ be the prototype of $\multi{}$ for $l$. Let $0_\eta$ be the origin of the component of the source corresponding to the singularity $\eta$, cf. Proposition \ref{pro:proto_multi}. We have  $0_\eta\in \overline{\Slocus{\zeta,\mzeta}{\proto{\multi{}}{l}}}$.
	\end{enumerate}

\end{df}
\begin{pro}
     Let $(\eta,\multi{})$ and $(\zeta,\mzeta)$ be S-multisingularities that occur for $l \ge 1$. Suppose that $(\eta,\multi{})\le_l(\zeta,\mzeta)$. Then $(\eta,\multi{})\le_{l+1}(\zeta,\mzeta)$.
\end{pro}
\begin{pro} \label{pro:splitsS}
	Let $(\eta,\multi{})$ and $(\zeta,\mzeta)$ be S-multisingularities that occur for $l\ge 1$. Introduce an order on the multisingularity $\multi{}=(\eta_1,\dots,\eta_k)$, such that $\eta=\eta_1$. We have $(\eta,\multi{}) \le_l (\zeta,\mzeta)$ if and only if there exist multisingularities $\mzeta_1,\dots,\mzeta_{k}$ such that:
    \begin{enumerate}
        \item The multisingularities $\mzeta_i$ sum up to $\mzeta$, i.e. $\mzeta=\mzeta_1+\dots+\mzeta_{k}$.
        \item For any index $i\ge 2$ we have $\eta_i \le_l \mzeta_i$ in the target stable hierarchy.
        \item We have $\zeta\in\mzeta_1$ and $(\eta,\{\eta\}) \le_l (\zeta,\mzeta_1)$ in the source stable hierarchy.
    \end{enumerate}
    \end{pro}
    Conditions $(1)$ and $(2)$ of the above statement depend only on combinatorics, as well as the stable hierarchy in the target. The proposition below allows us to rephrase also the third condition in terms of the stable hierarchy in the target.
    \begin{pro} 
    \label{pro:SandT}
        Let $\eta$ be a singularity and $(\zeta,\mzeta)$ an S-multisingularities that occur for $l\ge 1$. We have
        $$(\eta,\{\eta\}) \le_l (\zeta,\mzeta) \iff \eta\le_l \mzeta\,.$$
    \end{pro}
    \begin{proof}
        Let $f$ be a stable map of relative dimension $l$ . \\
        Suppose that $(\eta,\{\eta\}) \le_l (\zeta,\mzeta)$ and $y\in\Tlocus{\eta}{f}$. The point $y$ has a unique preimage in $ \Slocus{\eta,\{\eta\}}{f}$, which is contained in $\overline{\Slocus{\zeta,\mzeta}{f}}$.
        $$y\in f\Big(\overline{\Slocus{\zeta,\mzeta}{f}}\Big) = \overline{f(\Slocus{\zeta,\mzeta}{f})}= \overline{\Tlocus{\mzeta}{f}} \,.$$
        This implies $\eta \le_l \mzeta$. \\
        For the other implication, suppose that $\eta \le_l \mzeta$ and
        $x\in\Slocus{\eta,\{\eta\}}{f}$. Let $y=f(x)$. We have
        $$y\in \Tlocus{\eta}{f} \subset \overline{\Tlocus{\mzeta}{f}}=\overline{f(\Slocus{\zeta,\mzeta}{f})}=f\Big(\overline{\Slocus{\zeta,\mzeta}{f}}\Big)\,.$$
        It follows that $y$ has preimage in $\overline{\Slocus{\zeta,\mzeta}{f}}$. The point $x$ is the unique preimage of $y$, so $x\in \overline{\Slocus{\zeta,\mzeta}{f}}$.
    \end{proof}
\begin{cor} \label{cor:SandT}
       Propositions \ref{pro:splitsS} and \ref{pro:SandT} imply that the hierarchy in the target determines the hierarchy in the source.
\end{cor}
\begin{cor}
    Suppose that $(\eta,\multi{}) \le_l (\zeta,\mzeta)$. Then
    $$\multi{} \le_l \mzeta\,,\qquad \eta \le_l \zeta\,.$$
\end{cor} 
It is natural to ask whether the reverse implication in the above corollary holds. The answer is negative, as shown by the example below.
\begin{ex} \label{ex:SandT}
     Consider $l=1$ and the singularities
    $$
    \multi{}=A_1I_{22} \,,\qquad
    \eta=I_{22} \,,\qquad
    \mzeta=A^2_0A_3 \,,\qquad
    \zeta=A_0 \,.
    $$
    We have $\eta\le_1\zeta$. Moreover $\multi{}\le_1\mzeta$, because $I_{22}\le_1A_3$ and $A_1\le_1A_0^2$. Yet,
    $$(I_{22},A_1I_{22}) \nleq_1 (A_0,A^2_0A_3)\,. $$
    To see this, apply Proposition \ref{pro:splitsS}. The relation $(I_{22},A_1I_{22}) \le_1 (A_0,A^2_0A_3)$ would require a decomposition $A^2_0A_3=\mzeta_1+\mzeta_2$, 
    with $A_0\in \mzeta_1$, $I_{22}\le_1 \mzeta_1$ and $A_1\le_1 \mzeta_1$. No such decomposition exists, since
    $$
    I_{22}\nleq_1 A_0A_3 \Rightarrow A_3\notin \mzeta_1\,,
    \qquad
    A_1\nleq_1 A_3 \Rightarrow A_3\notin \mzeta_2\,.
    $$
\end{ex}

\section{Hierarchy in Mather nice dimensions}
In the Mather nice dimensions we can use Thom polynomials to determine the stable hierarchy.
\begin{thm}
\label{thm:hierarch_Tp}
    	Let $\multi{}$ and $\mzeta$ be multisingularities that occur for a positive integer $l$. Suppose that $\multi{}$ is a Mather multisingularity. The following conditions are equivalent: 
	\begin{enumerate}
		\item $\multi{} \le_l \mzeta$
		\item  The singularity locus $\Tlocus{\mzeta}{\proto{\multi{}}{l}}$ is nonempty.
		\item Consider any positive $\CC^*$-action that stabilizes the prototype $\proto{\multi{}}{l}$. The evaluation of the target Thom polynomial does not vanish in $\CC^*$-equivariant cohomology:
        $$\ThT{\mzeta}(\proto{\multi{}}{l})\neq 0 \in \coh^*_{\CC^*}(pt).$$  
	\end{enumerate}
\end{thm}
\begin{proof}
We have  $\multi{} \le_l \mzeta$ if and only if the origin $0$ lies in the closure of the singularity locus $\Tlocus{\mzeta}{\proto{\multi{}}{l}}$.
The prototypes of Mather multisingularities admit a positive $\CC^*$-action, cf. Theorem \ref{thm:quasihomogenous}. Every invariant nonempty closed set contains zero, thus conditions (1) and (2) are equivalent. The equivalence of (2) and (3) follows from \cite[Thm. 4.3]{FeherPatakfalvi}.
\end{proof}

The strength of Theorem~\ref{thm:hierarch_Tp}, together with the interpolation method of Section~\ref{sec:interpolation}, is that when both $\multi{}$ and $\mzeta$ are Mather multisingularities condition (3) may be checked algorithmically. That is, deciding the hierarchy of Mather multisingularities is algorithmic! 

We implemented this algorithm in a computer algebra software. That is, we calculated the Thom polynomials $\ThT{\mzeta}$ using the interpolation method, Theorem \ref{thm:interp2}. Then we computed the substitutions $\ThT{\mzeta}(\proto{\multi{}}{l})$ using the weights of positive $\CC^*$-actions collected on \cite{TPP}.

We illustrate the outcome as follows:
\begin{itemize}
    \item We present the Hasse diagram of multisingularities of dimension $\leq 5$ in Figure~\ref{fig:galaxy}. (The extension of this diagram to 6-dimensional Artin algebras would mean adding another 51 vertices, which is too many to display it here.)
    \item We present the hierarchy of (mono)singularities of dimension 6 in Figure~\ref{fig:Hierarchy_of_mono_6}.
    \item We present the elementary splittings of most (mono)singularities of dimension at most 6 in Figure~\ref{fig:splits}. Elementary splittings determine the poset structure due to Proposition \ref{pro:splits}.
\end{itemize}

\begin{rem}
    Recently, degenerations of {\em concise tensors} of dimension $5$ were established in \cite{JJtensors}. That partial order is an extension of the relation
    $\le_\infty$ for algebras of dimension~5, cf. ~Corollary~\ref{cor:Hilb0}. Hence, our 
    Figure \ref{fig:galaxy} and \cite[Fig.~4.1]{JJtensors} have a common subgraph. Order $\le_\infty$ for algebras of dimension~5 was also announced in \cite{Mazzola}.
\end{rem}
\begin{figure}[htp]
    \centering
\includegraphics[width=12cm]{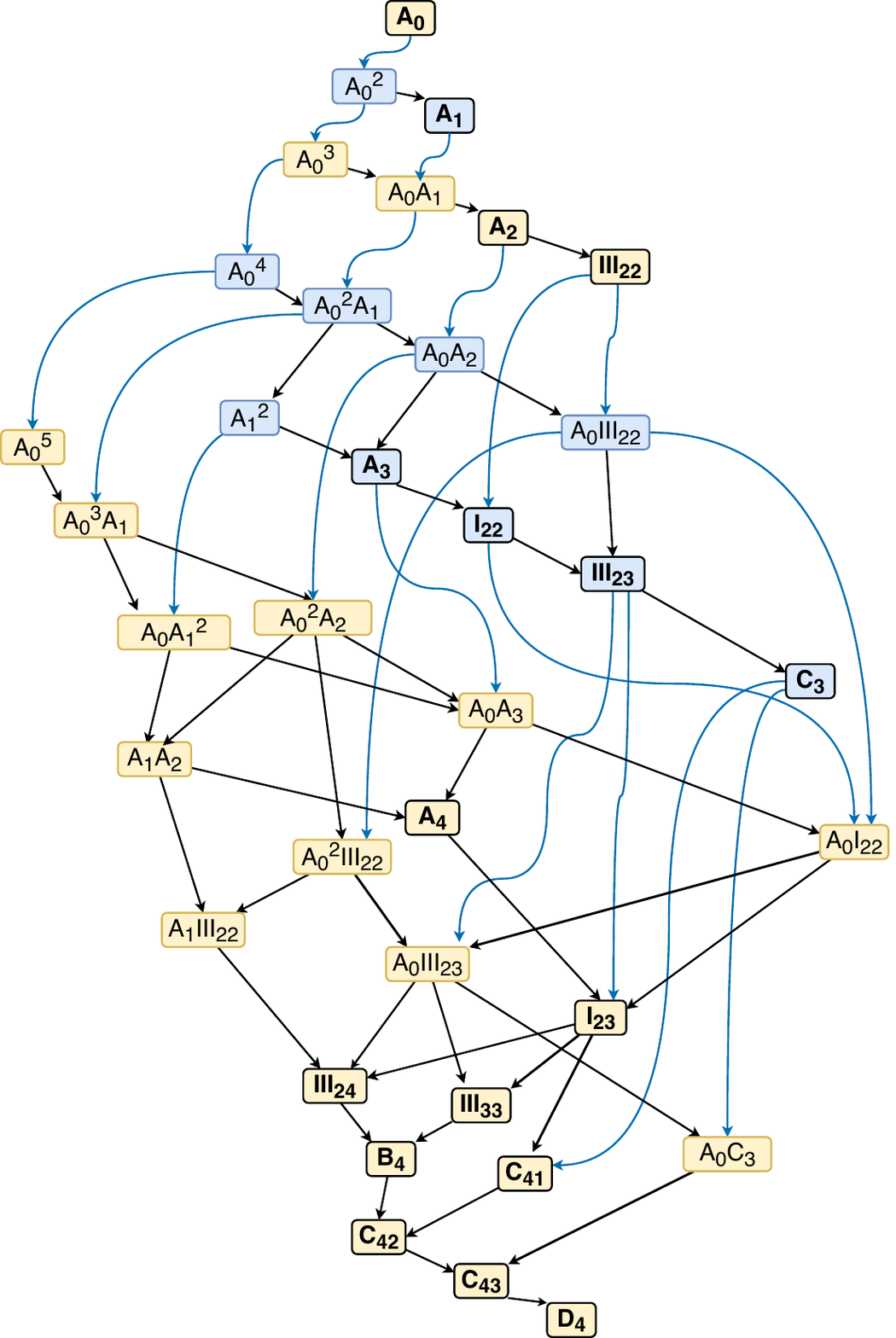}
    \caption{The Hasse diagram of the $\leq_l$-poset of multisingularities (Artin algebras) up to algebra dimension 5, for small values of $l$, and conjecturally for all $l$. Vertices in boldface are monosingularities (local algebras). The alternating yellow-blue colors indicate algebra dimensions $1,2,\ldots,5$.  Curved arrows in blue indicate adjacencies between different dimensional algebras. The extension of this diagram to dimension $6$ algebras would be too large to fit here, but it is determined by the data in Figure~\ref{fig:splits}.}
    \label{fig:galaxy}
\end{figure}

\begin{figure}[htp]
    \centering
\includegraphics[width=8cm]{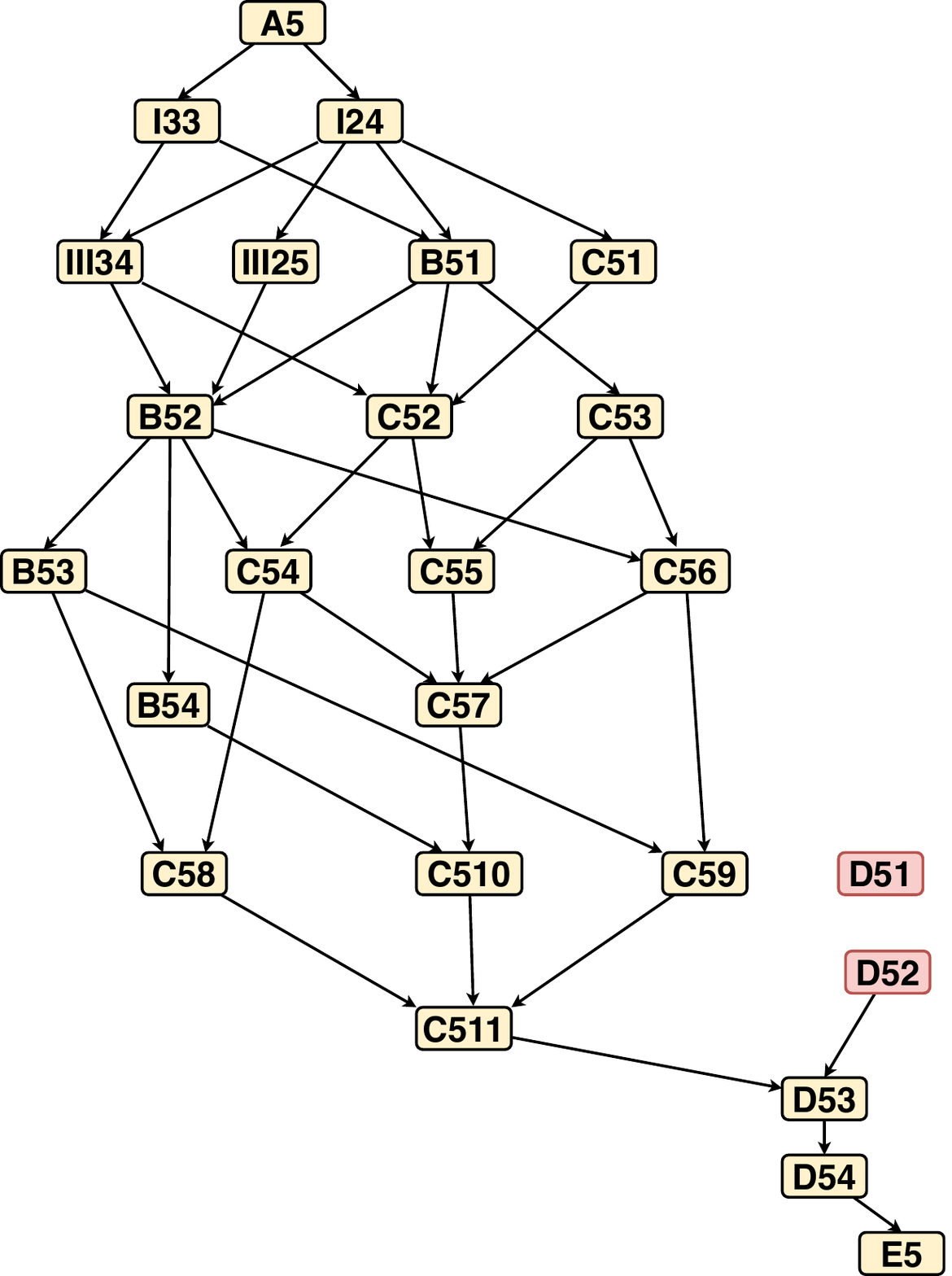}
    \caption{The Hasse diagram of the $\leq_{\infty}$-poset of monosingularities (local algebras) of dimension 6. For notation see Figure~\ref{Fig:List of 6dim algebras}. 
    To decide what arrows point into vertices D51 and D52 we would need more computing power, hence we left that part out. (C56 degenerates to  D52, but it may or may not be a covering relation.)}
\label{fig:Hierarchy_of_mono_6}
\end{figure}

\begin{rem}
An interesting phenomenon, illustrated in Figures~\ref{fig:galaxy} and \ref{fig:splits}, is that certain {\em mono}singularities cannot split directly. Instead, they must first transform into another {\em mono}singularity before decomposing into a genuine {\em multi}singularity, that is, into two or more local algebras. 
We refer to this behavior as a `dressing' phenomenon (by analogy with terminology from particle physics). For example, the singularity $C_{41}$ cannot split immediately: it must pass through chains like
\[
C_{41} \leftarrow I_{23} \leftarrow A_4 \leftarrow (A_1+A_2),
\qquad
C_{41} \leftarrow C_3 \leftarrow \III_{23} \leftarrow (A_0+\III_{22}),
\]
but never splitting immediately.
\end{rem}

\section{Hilbert schemes of points}\label{s:Hilb}
In the previous sections, we discussed a partial order on finitely dimensional algebras. In algebraic geometry, deformations of algebras of a given dimension are controlled by the Hilbert scheme of points. For a smooth variety $X$ the Hilbert scheme of $n$-points, denoted $\Hilb_n(X)$, parametrizes $0$-dimensional subschemes of $X$ of length $n$, see e.g. \cite{Nak_lectures}.
As a set:
$$
\Hilb_n(X)=\{Z\subset X| \dim Z=0\,,\, \dim_\CC(\O_Z(Z))=n\}
$$
where $\dim_\CC$ denotes the dimension of a vector space over $\CC$. From this point on, we will refer to $\dim_\CC$
 as the rank. To an algebra $A$ of rank $n$ we associate a locally closed stratum $\Strato{A}{X} \subset \Hilb_n(X)$. Intuitively, it parametrizes subschemes whose algebra is isomorphic to $A$. A closed point $p_Z\in \Hilb_n(X)$, corresponding to a subscheme $Z\subset X$, lies in $\Strato{A}{X}$ if and only if
	$$ \O_Z(Z)\simeq A\,.$$  We denote by $\Strat{A}{X}$ the closure of  the stratum $\Strato{A}{X}$. We present a formal definition of the stratum $\Strato{A}{X}$ in terms of a representing functor in Appendix \ref{s:deformation}.
\begin{rem}
    For $A=\prod_{i=1}^n\CC$ the closed stratum $\Strat{A}{X}$ is the smoothable component of the Hilbert scheme, see e.g. \cite{JJdok}, \cite[Sect. 5.3]{JB}. 
    For general $A$, these strata were considered in \cite[Sect. 4]{TOmult}. They  are examples of invariant manifolds, studied in \cite[Sect. 3]{Gaf}. 
\end{rem}

It turns out that strata $\Strato{A}{X}$ are closely related to the singularity loci of stable maps. In particular their relative position determine the stable hierarchy $\le_\infty$ for algebras of fixed rank. Their properties often boil down to classical facts from deformation theory.  Unfortunately, many such results are folklore knowledge for that community. They are well-known to experts, but the proofs, or even statements, are not written down. For the sake of completeness, we list facts we need below and present their proofs in Appendix~\ref{s:deformation}. 

\begin{pro} \label{pro:dim}
	Let $A$ be an algebra of rank $n$.
    The stratum  $\Strato{A}{X}$ is smooth of dimension
	$$
	\dim \Strato{A}{X}=n\dim X - \dim \Der(A,A)\,,
	$$
	where $\Der(A,A)$ is the vector space of derivations of $A$.
\end{pro}
\begin{pro} \label{pro:strat1}
	The stratification of $\Hilb_n(X)$ satisfies the frontier condition. For any algebras $A$ and $B$ of rank $n$ we have
    $$
    \Strato{A}{X}\cap \Strat{B}{X}\neq \varnothing \Rightarrow \Strato{A}{X}\subseteq \Strato{B}{X}\,.
    $$
\end{pro}

\begin{pro} \label{pro:deform_v2}
Let $A$ and $B$ be algebras of rank $n$. Suppose that for a smooth variety $X$ we have  $\Strato{A}{X}\subset \Strat{B}{X}$ and $\Strato{A}{X}\neq \varnothing$. Then for every smooth variety $Y$ we have
$\Strato{A}{Y}\subset \Strat{B}{Y}\,.$
\end{pro}
\begin{df}
    Let $A$ and $B$ be algebras of rank $n$. We write
    $$
    A\le_{\Hilb}B
    $$
    if there exists a smooth variety $X$, such that $\Strato{A}{X}\subset \Strat{B}{X}$ and $\Strato{A}{X}\neq \varnothing$.
\end{df}
\noindent
Proposition \ref{pro:deform_v2} implies that the relation $\le_{\Hilb}$ is a partial order on the set of algebras of rank $n$.
\begin{figure}
 \begin{center}
    \begin{tabular}{|l|lllllllllllllll} 
    \hline
$A_1$ & $A_0^2$ \\
 \hline 
$A_2$ & $A_0A_1$ \\
$\III_{22}$ & $A_2$ \\
 \hline 
$A_3$ & $A_1^2$ & $A_0A_2$ \\
$I_{22}$ & $\III_{22}$ & $A_3$ \\
$\III_{23}$ & $I_{22}$ & $A_0\III_{22}$ \\
$C_3$ & $\III_{23}$ \\
    \hline
   $A_4$ & $A_1A_2$ & $A_0A_3$\\
    $I_{23}$ & $\III_{23}$ &  $A_4$ & $A_0I_{22}$ \\
    $\III_{24}$ & $I_{23}$ & $A_1\III_{22}$ & $A_0\III_{23}$  \\
    $\III_{33}$ & $I_{23}$ & $A_0\III_{23}$ \\
    $B_4$ & $\III_{33}$ & $\III_{24}$ \\
    $C_{41}$ & $I_{23}$ & $C_3$ & & &\multicolumn{3}{l}{(only $I_{23}$ for $l=2$)} \\
    $C_{42}$ & $C_{41}$ & $B_4$ \\
    $C_{43}$ & $C_{42}$ & $A_0C_3$ \\
    $D_4$ & $C_{43}$ \\
    \hline
$A_5$ & $A_2^2$ & $A_1A_3$ & $A_0A_4$ \\
$I_{24}$ & $\III_{24}$ & $A_5$ & $A_1I_{22}$ & $A_0I_{23}$ \\
$I_{33}$ & $\III_{33}$ & $A_5$ & $A_0I_{23}$ \\
$\III_{25}$ & $I_{24}$ & $A_2\III_{22}$ & $A_1\III_{23}$ & $A_0\III_{24}$ \\
$\III_{34}$ & $I_{33}$ & $I_{24}$ & $A_1\III_{23}$ & $A_0\III_{33}$ & $A_0\III_{24}$ \\
$B_{51}$ & $B_4$ & $I_{33}$ & $I_{24}$ \\
$B_{52}$ & $B_{51}$ & $\III_{34}$ & $\III_{25}$ & $A_0B_4$ \\
$B_{53}$ & $B_{52}$ & $\III_{22}^2$ \\
$B_{54}$ & $B_{52}$ \\
$C_{51}$ & $C_{42}$ & $I_{24}$ & $A_0\III_{33}$ & & 
\multicolumn{3}{l}{($B_4$ instead $C_{42}$ for $l=1$)}\\
$C_{52}$ & $C_{51}$ & $B_{51}$ & $\III_{34}$ & $A_0B_4$ \\
$C_{53}$ & $C_{42}$ & $B_{51}$ & $A_0C_{41}$ \\
$C_{54}$ & $C_{52}$ & $B_{52}$ \\
$C_{55}$ & $C_{53}$ & $C_{52}$ & $A_0C_{42}$ \\
$C_{56}$ & $C_{53}$ & $B_{52}$ & $A_0C_{42}$ \\
$C_{57}$ & $C_{56}$ & $C_{55}$ & $C_{54}$\\
\hline
$A_6$ & $A_2A_3$ & $A_1A_4$ & $A_0A_5$ \\
$I_{25}$ & $\III_{25}$ & $A_6$ & $A_2I_{22}$ & $A_1I_{23}$ & $A_0I_{24}$ \\
$I_{34}$ & $\III_{34}$ & $A_6$ & $A_1I_{23}$ & $A_0I_{33}$ & $A_0I_{24}$ \\
$\III_{26}$ & $I_{25}$ & $A_3\III_{22}$ & $A_2\III_{23}$ & $A_1\III_{24}$ & $A_0\III_{25}$ \\
$\III_{35}$ & $I_{34}$ & $I_{25}$ & $A_2\III_{23}$ & $A_1\III_{33}$ & $A_0\III_{34}$ & $A_0\III_{25}$ \\
$\III_{44}$ & $I_{34}$ & $A_1\III_{24}$ & $A_0\III_{34}$ \\
$B_6 (l=2)$ & $B_{52}$ & $I_{34}$ & $I_{25}$ & $A_0B_{51}$ \\
$C_6 (l=1)$ & $C_{52}$ & $B_{52}$ & $I_{34}$ & $I_{25}$ & $A_1\III_{33}$ & $A_0C_{51}$ & $A_0\III_{34}$  \\
$C_6 (l= 2)$ & $C_{56}$ & $C_{55}$ & $I_{34}$ & $I_{25}$ & $A_1\III_{33}$ & $A_0C_{51}$ & $A_0\III_{34}$ \\
\hline
    \end{tabular}
\end{center}
    \caption{Elementary splittings of some local algebras for small values of $l$, and conjecturally for all $l$. All local algebras are included with corresponding singularity having codimension $\leq 5l+11$. Singularities $C_{58},\ldots, C_{5,11}$ are presently outside of the reach of our computing power. The cases for $C_{41}$, $C_{51}$, $C_6$ are for innocent reasons: for small $l$ some algebras in the `natural' elementary splitting do not occur yet.}
    \label{fig:splits}
\end{figure}

\section{Comparison of partial orders}
\label{sec:comparison}
\subsection{Comparison}
We prove that for algebras of fixed rank the partial orders $\le_{\Hilb}$ and the stable hierarchy $\le_\infty$ coincide. This result follows from the theorem presented below. The present subsection is devoted to the proof of the theorem, while the next one discusses its consequences.
\begin{thm} \label{thm:Hilb}
    Let $A$ and $B$ be algebras of rank $n$. Let $\multi{}$ and $\mzeta$ be the corresponding multisingularities. For any stable map $f:X\to Y$ such that $\Tlocus{\multi{}}{f}\neq \varnothing$ we have
    $$\Tlocus{\multi{}}{f}\subset \overline{\Tlocus{\mzeta}{f}}\iff \Strato{A}{X}\subset \Strat{B}{X}\,.$$
\end{thm}
\begin{cor} \label{cor:Hilb0} 
Let $A$ and $B$ be algebras of rank $n$. Let $\multi{}$ and $\mzeta$ be the corresponding multisingularities. We have
        $$\multi{} \le_\infty \mzeta \iff A\le_{\Hilb}B\,.$$   
\end{cor}
Theorem \ref{thm:Hilb} follows directly from Propositions \ref{pro:Hilb1}, \ref{pro:Hilb2} and \ref{pro:strat1}. We recall ideas of \cite{Gaf} and \cite{TOmult}. \\
Let $f:X\to Y$ be a finite map between smooth varieties. Its graph $\Gamma_f:X\to X\times Y$ is a regular embedding. By \cite[Def. 1.2]{Gaf}, or \cite[Pro. 3.2]{TOmult} the induces a map
$$i_f:\Hilb_n(X)\to \Hilb_n(X\times Y)\,.$$
is a regular embedding of codimension $\codim i_f =n\dim Y$. \\
Let $\mathcal{U}_X \subset \Hilb_n(X)\times X$ be the tautological family over $\Hilb_n(X)$. The diagonal embedding
$$\mathcal{U}_X \times Y \subset  \Hilb_n(X)\times X\times Y \subset \Hilb_n(X)\times X\times Y\times Y$$
induces a map
$$\Delta: \Hilb_n(X)\times Y \to \Hilb_n(X\times Y)\,.$$
Intuitively, it associates to a subscheme $Z\subset X$ and a point $y\in Y$ the same subscheme in the product $X\times\{y\}$;
cf. \cite[Pro. 1.4]{Gaf} and \cite[Pro. 3.8]{TOmult}. The relative Hilbert scheme $\Hilb_n(f)$ is defined as the fiber product of $i_f$ and $\Delta$.
Let $A$ be an algebra of rank $n$ and $\multi{}$ the corresponding multisingularity. We define a closed subset $C_\multi{}(f) \subset \Hilb_n(f)$ and a locally closed subset $C^o_\multi{}(f)\subset \Hilb_n(f)$ as fiber products, cf. \cite[Def. 4.4]{TOmult}:
\[\begin{tikzcd}
	{C^o_\multi{}(f)} & {C_\multi{}(f)} & {\Hilb_n(f)} & {\Hilb_n(X)} \\
	{\Strato{A}{X}\times Y} & {\Strat{A}{X}\times Y} & {\Hilb_n(X)\times Y} & {\Hilb_n(X\times Y)}
	\arrow[from=1-1, to=1-2]
	\arrow[from=1-1, to=2-1]
	\arrow[from=1-2, to=1-3]
	\arrow[from=1-2, to=2-2]
	\arrow[from=1-3, to=1-4]
	\arrow[from=1-3, to=2-3]
	\arrow["{i_f}"', from=1-4, to=2-4]
	\arrow[from=2-1, to=2-2]
	\arrow[from=2-2, to=2-3]
	\arrow["\Delta", from=2-3, to=2-4]
\end{tikzcd}\]
The projecton $\Hilb_n(X)\times Y\to Y$ induces a map $\pi:\Hilb_n(f) \to Y$. For a projective $X$ the map $\pi$ is proper. \\
For a general map the sets $C^o_\multi{}(f)$ and $C_\multi{}(f)$ may behave badly. We will show that for a stable map the set $C_\multi{}(f)$ is a closure of $C^o_\multi{}(f)$ and the image in $\pi(C_\multi{}(f))$ is nearly the singularity loci $\overline{\Tlocus{\multi{}}{f}}$, see Propositions \ref{pro:dense} and \ref{pro:image}. 
\begin{rem}
	The statement about the image of $C_\multi{}(f)$ in $\pi$ was announced, and proved for cohomology classes, in \cite[Rem. 4.5 and Pro. 4.15]{TOmult}.
\end{rem}
We recall Gaffney's notion of invariant subset \cite[Def. 3.2]{Gaf}. A subset $T\subset \Hilb_n(X)$ is called invariant if it is preserved by germs of biholomorphic maps on $X$. More precisely, let $p_Z\in T$ be a point corresponding to a subscheme $Z\subset X$. We demand that for any open subsets $U,U'\subset X$ such that $Z\subset U$ and isomorphism $h:U\to U'$ the point $p_{h(Z)}$ also lies in $A$. The following properties of invariant subsets are straightforward:
\begin{pro}
	Suppose that a subset $T\subset \Hilb_n(X)$ is invariant
	\begin{enumerate}
		\item The closure $\overline{T}$ and boundary $\overline{T}-T$ are invariant subsets.
		\item Suppose that $T$ is closed. Let $T^{sm}$ be the smooth part of $T$ and $T^{sm,top}$ union of its top dimensional components. Both  $T^{sm}$ and $T^{sm,top}$ are invariant subsets.
	\end{enumerate}
\end{pro}
\begin{cor} \label{cor:inv}
	Suppose that $M\subset \Hilb_n(X)$ is an invariant locally closed subset. Then its boundary $\overline M-M$ can be stratified into finitely many invariant smooth pure-dimensional locally closed subsets.
\end{cor}
Gaffney characterized stable maps in terms of invariant manifolds of Hilbert scheme, see \cite[Thm. 3.5 and Thm. 3.8]{Gaf}. He proved that for a stable map $f$ and invariant smooth manifold $M\subset \Hilb_n(X)$
the intersection of $\Delta(M\times Y)$ with $i_f$ is smooth of pure dimension \hbox{$\dim M+\dim Y-\codim(i_f)$.} In particular for $M=\Strato{A}{X}$ we obtain
\begin{align}\label{w:Cdim}
\dim C^o_\multi{}(f)=\dim \Strato{A}{X}+\dim Y-\codim(i_f)\,.
\end{align}
Denote this number by $c_\multi{}$.
\begin{pro} \label{lem:dim}
	Let $\multi{}$ be a multisingularity. All the components of $C_\multi{}(f)$ have dimension at least $c_\multi{}$.
\end{pro}
\begin{proof}
	The set $\Strato{A}{X}$ is smooth of pure dimension. Therefore its closure $\Strat{A}{X}$ is also pure dimensional of the same dimension. The conclusion follows from \cite[Lem. 43.3.12]{Stacks}.
\end{proof}
\begin{pro} \label{pro:dense}
	Suppose that the map $f$ is stable. For any multisingularity $\multi{}$ the set $C^o_\multi{}(f)$ is open dense subset of $C_\multi{}(f)$.
\end{pro}
\begin{proof}
	$C^o_\multi{}(f)$ is open in $C_\multi{}(f)$ because $\Strato{A}{X}$ is open in $\Strat{A}{X}$. \\
	By Proposition \ref{lem:dim}  it is enough to show that the  dimension of the complement of $C^o_\multi{}(f)$ is smaller than $c_\multi{}$. By Corollary \ref{cor:inv} we can stratify the complement $\Strat{A}{X}-\Strato{A}{X}$ into finitely many smooth pure-dimensional invariant strata $S_i$, i.e.
	$$\Strat{A}{X}-\Strato{A}{X}=\bigcup S_i$$
	Intersecting with $i_f$ we obtain
	$$C_\multi{}(f)-C^o_\multi{}(f)=\bigcup i_f^*(S_i\times Y)\,.$$
	By Gaffney characterization of stable maps \cite[Thm. 3.5]{Gaf} we have
	$$\dim i_f^*(S_i\times Y)=\dim S_i +\dim Y -\codim(i_f) < \dim \Strato{A}{X} +\dim Y-\codim(i_f)=c_\multi{}\,.$$ 
	There are finitely many strata, thus the dimension of their union is also smaller than $c_\multi{}$.
\end{proof}

\begin{pro} \label{pro:image}
	Let $\multi{}$ be a multisingularity of dimension $n$. Consider the open subset
	$$Y_{\le n}=\{y\in Y| \dim \multi{y}\le n\}\,.$$
	\begin{enumerate}
		\item We have
		$$
		\pi(C^o_\multi{}(f))\cap Y_{\le n}=\Tlocus{\multi{}}{f}\,,
		$$
        The restriction $\pi: C^o_\multi{}(f) \cap \pi^{-1}(Y_{\le n}) \to \Tlocus{\multi{}}{f}$ is 1-1.
		\item Suppose that $f$ is stable. Then
		$$
		\pi(C_\multi{}(f)) \cap Y_{\le n}\subset \overline{\Tlocus{\multi{}}{f}} \cap Y_{\le n}\,,
		$$
        Moreover, for projective $X$ we have an equality.
	\end{enumerate}
\end{pro}
\begin{proof}
    Let $A$ be the algebra corresponding to singularity $\multi{}$.
	For the first point consider a point $p_Z\in \Hilb_n(X)$ corresponding to a subscheme $Z\subset X$. By \cite[Pro. 1.4]{Gaf} the point $(p_Z,y)$ lies in $C^o_\multi{}(f)$ if and only if $p_Z$ lies in $\Strato{A}{X}$ and $Z \subseteq f^{-1}(y)$. For $y \in Y_{\le n}$ the fiber $f^{-1}(y)$ has at most the same rank as $Z$, thus
	$$
	Z \subseteq f^{-1}(y) \iff Z=f^{-1}(y)\,.
	$$
	The second point follows from the first, Proposition \ref{pro:dense} and continuity of the map $\pi$. \\
	The "moreover" part follows from the fact that for a projective $X$ the map $\pi$ is proper.
\end{proof}
\begin{pro} \label{pro:Hilb1}
	Let $A$ and $B$ be algebras of rank $n$. Let $\multi{}$ and $\mzeta$ be the corresponding multisingularities. Suppose that the map $f$ is stable and $\Strato{A}{X}\subset \Strat{B}{X}$. Then
	$$\Tlocus{\multi{}}{f}\subset\overline{\Tlocus{\mzeta}{f}}.$$
\end{pro}
\begin{proof}
	Let $y\in \Tlocus{\multi{}}{f}$. The fiber $f^{-1}(y)$ corresponds to a point $p_{f^{-1}(y)} \in \Strato{A}{X}$. We know that $\Strato{A}{X} \subseteq \Strat{B}{X}$, thus
	$$(p_{f^{-1}(y)},y) \in C^o_\multi{}(f)\subseteq  C_\mzeta(f)\,. $$ 
	Therefore, $y\in \pi(C_\mzeta(f))$. The map $f$ is stable and $y\in Y_{\le n}$, so Proposition \ref{pro:image} (2)  implies that $y\in \overline{\Tlocus{\mzeta}{f}}$.
\end{proof}
\begin{pro} \label{pro:Hilb2}
	Let $A$ and $B$ be algebras of rank $n$. Let $\multi{}$ and $\mzeta$ be the corresponding multisingularities. Suppose that $\Tlocus{\multi{}}{f}\cap \overline{\Tlocus{\mzeta}{f}}\neq \varnothing$. Then
	$$\Strato{A}{X}\cap \Strat{B}{X}\neq\varnothing\,.$$
\end{pro}
\begin{proof}
    Consider the subset $Y_n\subset Y$ of points with fiber of rank $n$, i.e.
	$$
	Y_n=\{y\in Y| \dim(\multi{y})=n\}\,.
	$$
	This set is locally closed. We consider a reduced scheme structure on it. Let $\Delta:Y\to Y\times Y$ be the diagonal map. Schemes $Z$ and $Z_n$ are defined as pullbacks:
	$$
	\begin{tikzcd}
		{Z_n} & Z & Y \\
		{Y_n\times X} & {Y\times X} & {Y\times Y}
		\arrow[from=1-1, to=1-2]
		\arrow[from=1-1, to=2-1]
		\arrow[from=1-2, to=1-3]
		\arrow[from=1-2, to=2-2]
		\arrow["\Delta", from=1-3, to=2-3]
		\arrow[from=2-1, to=2-2]
		\arrow["{\id_Y\times f}"', from=2-2, to=2-3]
	\end{tikzcd}
	$$
	We treat $Z_n$ as a family of subschemes of $X$ parametrized by $Y_n$. The fiber of $Z_n$ over $y\in Y_n$ is the subscheme $f^{-1}(y)$. By definition of $Y_n$ the family $Z_n$ has a constant Hilbert polynomial, so
    it induces a morphism $$\rho:Y_n \to \Hilb_n(X).$$
    We have $\rho(\Tlocus{\multi{}}{f}) \subset \Strato{A}{X}$ and $\rho(\Tlocus{\mzeta}{f}) \subset \Strato{B}{X}$.
	Proposition follows from the continuity of $\rho$.
\end{proof}

\subsection{Consequences}
\begin{cor} \label{cor:order}
    Let $A$ and $B$ be algebras of rank $n$. Let $\multi{}$ and $\mzeta$ be the corresponding multisingularities.
        Suppose that $A\le_{\Hilb}B$ (equivalently $\multi{} \le_\infty \mzeta$) and $\multi{}$ occurs for $l$. Then $\mzeta$ occurs for $l$ and $\multi{} \le_l \mzeta$.
\end{cor}
\begin{proof}
It is a direct consequence of Theorem \ref{thm:Hilb}.
\end{proof}

\begin{cor}\label{cor:Der}
    Let $A$ be an algebra of rank $n$ and $\multi{}$ the corresponding multisingularity. We have
    $$
    \tcodim_l(\multi{})=nl+\dim \Der(A,A)\,.
    $$
\end{cor}
\begin{proof}
    Let $f: X\to Y$ be a stable map of codimension $l$ such that $\Tlocus{\multi{}}{f}\neq \varnothing$. Let $m=\dim X$ and $m+l=\dim Y$.
    Proposition \ref{pro:image} imply that
    $$\dim C^o_\multi{}(f) =\dim \Tlocus{\multi{}}{f}\,.$$
    Therefore Proposition \ref{pro:dim} implies that
    \begin{align*}
    \tcodim_l(\multi{})
    &=\dim Y -  \dim C^o_\multi{}(f) \\
    &=\dim Y - (\dim \Strato{A}{X}+\dim Y -\codim i_f) \\
    &=n\cdot \dim Y -\dim \Strato{A}{X} \\
    &=n(m+l)-mn+\dim\Der(A,A) \\
    &=nl+\dim\Der(A,A)
    \end{align*}
\end{proof}
\begin{cor} \label{cor:Gor}
     Let $A$ and $B$ be algebras of rank $n$. Let $\multi{}$ and $\mzeta$ be the corresponding multisingularities. Suppose that $\mzeta\le_\infty\multi{}$ and $B$ is Gorenstein. Then $A$ is Gorenstein.
\end{cor}
\begin{proof}
    This follows from Corollary \ref{cor:Hilb0} and the fact that Gorenstein locus is open in the Hilbert scheme of points, see \cite[Lem. 5.15 (ii)]{JB}.
\end{proof}
    \begin{ex}\label{ex:Gor}
    In the above corollary assumption $\dim A=\dim B$ is necessary. We have $C_{41}\le_\infty C_3$. The local algebra corresponding to $C_{41}$ is Gorenstein, while the one corresponding to $C_3$ is not, see \cite[Lem. 2.25]{JB}.
    \end{ex}
\begin{rem}
   The relation $\le_{\Hilb}$ may be used to obtain information about stable hierarchy only for multisingularities of the same dimension. Without the assumption $\dim\multi{}=\dim\mzeta$ Corollary \ref{cor:order} and \ref{cor:Gor} do not hold, see Examples \ref{ex:C3C41} and \ref{ex:Gor}. \\
   A variant of Proposition \ref{pro:splits}
   for the relation $\le_{\Hilb}$ is well known in the deformation theory community. Proposition \ref{pro:splits} is a more general fact as we do not assume that $\dim\multi{}=\dim\mzeta$.
\end{rem}

\section{Appendix: Deformation theory} \label{s:deformation}
Let $n$ be a positive integer and $A$ an algebra of rank $n$. We denote by $\G$ its group of automorphisms.
\subsection{Fppf bundles}
\begin{df} Let $E\,,P$ and $S$ be schemes over $\CC$.
    \begin{itemize}
        \item A map $\pi:E\to S$ is called a $\T$--bundle if it is locally trivial in the fppf topology with fiber $\T$.
        \item A map $\pi:P\to S$ is called a $\G$--bundle if the group $\G$ acts on $P$ and the map $\pi$ is equivariantly locally trivial in the fppf topology with fiber $\G$.
    \end{itemize}
\end{df}
Let us recall standard facts about fppf bundles.
\begin{pro}\label{ap:fppf}
Let $E\,,P$ and $S$ be schemes over $\CC$.
    \begin{itemize}
        \item Suppose that a map $E\to S$ is a $\T$--bundle or a $\G$--bundle. Then it is faithfully flat.
        \item  Let $E \to S$ be a faithfully flat morphism. It is a $\T$--bundle if and only if there exists an isomorphism: 
        $$E\times_SE\simeq \T \times E\,.$$
        \item Let $\pi:P \to S$ be a faithfully flat morphism. Suppose that the group $\G$ acts on $P$ fiberwise. The map $\pi$ is a $\G$--bundle if and only if the natural map: 
        $$ G \times P \to P\times_SP\,$$
        is an isomorphism.
    \end{itemize}
\end{pro}
 There is an equivalence between $G$--bundles and $\Spec A$--bundles over a given base scheme $S$. Let $P\to S$ be a $G$--bundle. It induces a $\T$--bundle defined by the mixed product
$\T\times_{\G} P$.
Conversely to a $\T$--bundle we associate a $\G$--bundle defined by the isomorphisms scheme
$\Iso(\T\times S,P)$.
\begin{pro} \label{ap:bundles}
    The two above construction provide a mutually reverse bijections between isomorphism classes of $\T$--bundles and $\G$--bundles over $S$.
\end{pro}
\begin{proof}
    Let $P\to S$ be a $\G$--bundle and $E\to S$ a  $\T$--bundle. There is a standard map
    \begin{align} \label{w:Ap1}
    \T\times_{\G}\Iso(\T\times S,E)\to E\,.
    \end{align}
    Consider composition of canonical isomorphisms
    $$
    (\T \times S) \times_S P \simeq \T\times P \simeq
    \T \times_{\G} \G\times P \simeq 
    (\T \times_{\G} P) \times_S P\,. 
    $$
    The last one follows from Proposition \ref{ap:fppf}. This composition  induces a map
    \begin{align} \label{w:Ap2}
        P \to \Iso(\T \times S, \T \times_{\G} P)\,.
    \end{align}
    It is enough to show that the maps \eqref{w:Ap1} and \eqref{w:Ap2} are isomorphisms. This is obvious when the bundles $P$ and $E$ are trivial. Thus, these maps are isomorphisms on an fppf cover. By a faithfully flat descent they are isomorphisms.
\end{proof}

\begin{pro} \label{ap:trivial}
    Suppose that $\pi:E\to \Spec R$ is a $\T$--bundle.
    \begin{itemize}
        \item Suppose that $R$ is a field $\kk$. Let $\overline{\kk}$ be its algebraic closure. The pullback bundle $\pi\times_{\Spec\kk} \Spec\overline{\kk}$ is trivial.
        \item Suppose that $R$
        is a local Artinian algebra over $\CC$. Then, the bundle $\pi$ is trivial.
    \end{itemize}
\end{pro}
\begin{proof}
    By proposition \ref{ap:bundles} it is enough to prove the corresponding fact for a $\G$--bundle $P\to \Spec R$. \\
    Suppose that $R \simeq \overline{\kk}$ is an algebraically closed field. The pullback  $P\times_{\Spec\overline{\kk}} P\to P$ is a trivial bundle, cf. Proposition \ref{ap:fppf}. The group $G$ is of finite type, thus the trivial bundle is a finite type morphism. By a faithfully flat descent the map $P\to \Spec \overline{\kk}$ is also of a finite type. Thus, by Nullstellensatz, the scheme $P$ has a $\overline{\kk}$-point and the bundle $E\to S$ has a section. Therefore, it is trivial. \\
    Suppose now that $R$ is a local Artinian algebra over $\CC$. The group $\G$ is smooth, so the trivial bundle  $P\times_{\Spec R} P\to P$ is a smooth morphism. By faithfully flat descent the map $P\to \Spec R$ is also smooth. By the previous point the scheme $P$ has a $\CC$-point, so a section over $\Spec \CC$. By infinitesimal lifting criterion of smoothness this section can be extended to $\Spec R$, see diagram below:
    \[\begin{tikzcd}
	{\Spec \CC} & {P} \\
	{\Spec R} & {\Spec R} 
	\arrow[from=1-1, to=1-2]
	\arrow[from=1-1, to=2-1]
	\arrow[from=1-2, to=2-2]
	\arrow[dotted, from=2-1, to=1-2]
	\arrow["=", from=2-1, to=2-2]
\end{tikzcd}\]
\end{proof}
\subsection{Moduli space of algebras}
Let $V$ be a vector space of dimension $n$. The moduli space of algebras $\mathcal{M}_n $ parametrizes associative and unital operations on $V$, see \cite[Sect. 2]{JS} for an overview. It is an affine scheme 
$$
\mathcal{M}_n:=\Spec \CC[\lambda^k_{ij}|2 \le\, i, j\, \le n\,,\, 1 \le k \le n]/I_n
$$
where the ideal $I_n$ is defined by
$$I_n:=\Big(\sum_k\lambda^k_{ij}\cdot \lambda^m_{kl}
-
\sum_k\lambda^m_{ik}\cdot \lambda^k_{jl}
| 1 \le i, j, l, m \le n\Big)\,.
$$
Let $O_A\subset \mathcal{M}_n$ be the subset parameterizing algebra structures on $V$ which are isomorphic to $A$. This subset is an orbit of an action of the group $\GL(V)$ on $\mathcal{M}_n$, thus it is locally closed \cite[Thm. 2.14]{JS}.   \\
Consider a space
$$
\mathcal{U}:=\Spec
\frac{\CC[\lambda^k_{ij},x_1,\dots,x_n]}
{I_n+
(x_i\cdot x_j-\sum_k \lambda_{ij}^k\cdot x_k
|1\le i,j \le n
)}
$$
together with a map $p:\mathcal{U}\to\mathcal{M}_n$. Over a closed point $x\in O_A$ the fiber of the map $p$ is $\Spec A$. Moreover, the restriction of $p$ to the orbit $O_A$ is a  $\Spec A$--bundle. To see this, note that the map $\GL(V) \to O_A$ is faithfully flat. The pullback by this map gives a trivial bundle.

\subsection{Deformations of algebras}
All valuation rings are assumed to be $\CC$-algebras.
\begin{df} \label{df:deform}
	Let $B$ and $C$ be algebras of rank $n$ and $(R,\m)$ a valuation ring. We say that $C$ deforms to $B$ over $R$ if there exists an $R$-algebra $\B$ such that:
	\begin{itemize}
		\item $\B$ is a free $R$-module of rank $n$.
		\item Restriction of $\B$ to the closed point of $\Spec R$ is equal to $B$, i.e.
		$\B/\m\B\simeq B\,.$
		\item Over the generic point of $\Spec R$, the algebra $\B$ is isomorphic to the trivial extension of $C$, i.e.
		$$  \B\otimes_{R}\overline{(R)} \simeq C\otimes_\CC\overline{(R)}\,.$$
        Here $\overline{(R)}$ denotes the algebraic closure of the field of fractions.
	\end{itemize}
    We say that $C$ deforms to $B$ if there exists a valuation ring $R$ such that $C$ deforms to $B$ over $R$.
\end{df}
\begin{pro} [{\cite[Lem. 3.10]{JB}}]
    Let $B$ and $C$ be algebras of rank $n$. Algebra $C$ deforms to $B$ if and only if it deforms to $B$ over $\CC[[t]]$.
\end{pro}
\subsection{Representable functor}
For a smooth variety $X$ let $\StratF{A}{X}$ be a contravariant functor on the category of schemes defined by
$$
\StratF{A}{X}(T)=\Bigg\{
\begin{tikzcd}[ row sep=small]
	Z & {T\times X} \\
	& T
	\arrow[hook, from=1-1, to=1-2]
	\arrow[from=1-1, to=2-2]
	\arrow[from=1-2, to=2-2]
\end{tikzcd} \in \Hilb_n(X)(T)
\Bigg|
Z\to T \text{ is a } \Spec A\text{--bundle}
\Bigg\}
$$
\begin{thm} \label{ap:thm}
    The functor $\StratF{A}{X}$ is representable by a locally closed subscheme of $\Hilb_n (X)$.
\end{thm}
This theorem allows us to define stratum $\Strato{A}{X}$ in a way consistent with the intuition presented in Section \ref{s:Hilb}. We postpone its proof to the end of the section. 
\begin{df}
    Let $\Strato{A}{X} \subset \Hilb_n(X)$ be the subscheme representing $\StratF{A}{X}$. Let $\Strat{A}{X}$ be its closure in $\Hilb_n(X)$.
\end{df}
\begin{cor} \label{ap:tangent}
Let  $p_Z\in \Hilb_n(X)$ be a closed point, corresponding to a subscheme $Z\subset X$,
\begin{enumerate}
    \item The point $p_Z$ lies in $\Strato{A}{X}$ if and only if
	$\O_Z(Z)\simeq A\,.$
    \item Tangent space to $\Strato{A}{X}$ at $p_Z$ is a subspace of $T_{p_Z}\Hilb_n(X)$. Its elements are families $\mathcal{Z} \subseteq \Spec\CC[\varepsilon]/\varepsilon^2 \times X$ that induce trivial  deformations of the algebra $\O_Z(Z)$, i.e.
    $$
T_{p_Z}\Strato{A}{X}=\Bigg\{
\begin{tikzcd}[ row sep=small]
	\mathcal{Z} & {\Spec\CC[\varepsilon]/\varepsilon^2\times X} \\
	& \Spec\CC[\varepsilon]/\varepsilon^2
	\arrow[hook, from=1-1, to=1-2]
	\arrow[from=1-1, to=2-2]
	\arrow[from=1-2, to=2-2]
\end{tikzcd}
\Bigg|
\begin{array}{c}
     \O_\mathcal{Z}(\mathcal{Z})\simeq A\otimes_\CC \CC[\varepsilon]/\varepsilon^2  \\
    \mathcal{Z}\times_{\Spec \CC[\varepsilon]/\varepsilon^2} \CC= Z
\end{array}
\Bigg\}
$$
\end{enumerate}
\end{cor}
\begin{proof}
    It is a direct consequence of Proposition \ref{ap:trivial}.
\end{proof}

\subsection{Properties of the stratification}
First, we will show that the strata $\Strato{A}{X}$ satisfy the frontiers condition, cf. Proposition \ref{pro:strat1}. Intuitively, the relative position of the strata in the Hilbert scheme of points describes deformations of algebras embedded in a smooth variety $X$, whereas Definition \ref{df:deform} describes abstract deformations.
It turns out that the abstract deformations determine the embedded ones and govern the geometry of strata. This is a standard phenomenon in deformation theory. For an analogous results in the moduli space of algebras $\mathcal{M}_n$ see \cite{JS}. For the smoothable case in the Hilbert scheme, i.e. the case of algebra $\prod_{i=1}^N\CC$, see \cite[Pro. 5.6 and Thm. 3.16]{JB}, \cite[Pro. 2.1]{BB}, \cite[Lem. 2.2]{CN} and \cite[Lem. 4.1]{CEVV}.  \\
We start with the case when the variety $X$ is an affine space.
\begin{pro} \label{lem:strat}
		Let $A$ and $B$ be algebras of rank $n$ and $X=\Aa^m$. The following are equivalent:
		\begin{enumerate}
			\item $\Strato{A}{X}\cap \Strat{B}{X}\neq\varnothing$. 
			\item $A$ is generated by $m$ elements and $B$ deforms to $A$ in the sense of definition \ref{df:deform}.
			\item $\Strato{A}{X}\neq \varnothing$ and $\Strato{A}{X}\subseteq \Strat{B}{X}$. 
		\end{enumerate}
\end{pro}

\begin{proof}
	(1) $\Rightarrow$ (2):
	A closed point $p\in \Strato{A}{X}\cap \Strat{B}{X}$ corresponds to a surjection $\CC[x_1,\dots,x_m] \to A$, thus $A$ is generated by $m$ elements.
	The  stratum $\Strato{B}{X}$ is locally closed, so by \cite[Lem. 3.10]{JB} there exists a morphism $\Spec\CC[[t]] \to \Hilb_n(X)$ such that the general point is mapped to $\Strato{B}{X}$ and the closed point is mapped to the point $p\in \Strato{A}{X}$. By Proposition \ref{ap:trivial} this exactly means that $B$ deforms to $A$ in the sense of definition \ref{df:deform}. \\
	(2) $\Rightarrow$ (3):
	The algebra $A$ is generated by $m$ elements, so the stratum $\Strato{A}{X}$ is nonempty. \\
	Let $p \in \Strato{A}{X}$ be a point corresponding to a surjection $\varphi:\CC[x_1,\dots,x_m] \to A$. Let $\A$ be a deformation of $B$ to $A$ over the power series ring $\CC[[t]]$. Denote by $\pi:\A\to \A/t\A\simeq A$ the quotient morphism. Pick any preimages of $\varphi(x_1),\dots, \varphi(x_n)$ in $\A$. They induce a map of $\CC[[t]]$-modules
    $$\tilde{\varphi}:\CC[[t]][x_1,\dots,x_m] \to \A$$
    that lifts $\varphi$. This can be summarized in the following diagram:
	\[\begin{tikzcd}
		{\CC[[t]][x_1,\dots, x_m]} & \A \\
		{\CC[x_1,\dots, x_m]} & {\A/t\A}
		\arrow["{\tilde{\varphi}}", from=1-1, to=1-2]
		\arrow[from=1-1, to=2-1]
		\arrow["\pi", from=1-2, to=2-2]
		\arrow["\varphi", from=2-1, to=2-2]
	\end{tikzcd}\]
	The map $\tilde\varphi$ is a surjection due to the Nakayama lemma for the local ring $\CC[[t]]$. The algebra $\A$ is free $\CC[[t]]$-module of rank $n$, so $\tilde\varphi$ induces a map
	$$\widehat{\varphi}:\Spec{\CC[[t]]} \to \Hilb_n(X)\,.$$
	Let $z_\m$ be the closed point of $\Spec{\CC[[t]]}$ and $z_0$ the general point. Definition \ref{df:deform} implies that $\widehat{\varphi}(z_m)=z\in \Strato{A}{X}$ and  $\widehat{\varphi}(z_0)\in \Strato{B}{X}$. We have $z_m\in \overline{z_0}$. The result follows from the continuity of the map $\widehat{\varphi}$.
	\\
	(3) $\Rightarrow$ (1) is obvious.
\end{proof}
The above proposition implies Propositions \ref{pro:strat1} and \ref{pro:deform_v2} under the assumption that the varieties $X$ and $Y$ are affine spaces. The generalization to the arbitrary smooth varieties follows from the proposition below.

\begin{pro} \label{lem:complete}
    Let $A$ and $B$ be algebras of rank $n$. Let $X$ be a smooth variety of dimension $m$ and $Y=\Aa^m$.
    \begin{enumerate}
        \item  $\Strato{A}{X}\cap \Strat{B}{X}\neq\varnothing \iff  \Strato{A}{Y}\cap \Strat{B}{Y}\neq \varnothing\,,$
        \item  dimensions of tangent spaces to $\Strato{A}{X}$ and  $\Strato{A}{Y}$ are equal\,.
    \end{enumerate}
\end{pro}
\begin{proof}
    Let $p_Z \in \Strato{A}{X}$  be a closed point corresponding to a subscheme $Z\subset X$.
    Suppose that $A$ is local, i.e. the support of $Z$ consists of a single point $x=\supp(Z)$. The general case follows by considering each part separately. \\
    (1): By \cite[Lem. 3.10]{JB} the containment of strata is encoded by deformations of the map of local rings $\O_{X,x} \to \O_{Z,x}$ over the formal power series ring $\CC[[t]]$, cf. the proof of Lemma \ref{lem:strat}. The scheme $Z$ is zero dimensional, thus this map factorizes through completion
    $\O_{X,x} \to \widehat{\O}_{X,x} \to  \O_{Z,x}\,.$
    The power series ring $\CC[[t]]$ is complete, thus the deformations are determined on the level of completion. Completions of the local rings $\O_{X,x}$ and $\O_{\Aa^m,0}$ are isomorphic, which proves the claims. \\
    (2): Tangent space is the space of trivial deformations of the map  $\O_{X,x} \to \O_{Z,x}$ over dual numbers $\CC[\varepsilon]/\varepsilon^2$, cf. Corollary \ref{ap:tangent}. Analogously to the previous point, such deformations are defined on the level of completions. 
\end{proof}
Next, we prove that the stratum $\Strato{A}{X}$ is smooth and compute its dimension.
\begin{pro}
	Let $A$ be an algebra of rank $n$. The stratum  $\Strato{A}{X}$ is smooth.
\end{pro}

\begin{proof}[Proof of Proposition \ref{pro:dim}]
    We will show that the functor $\StratF{A}{X}$ satisfies infinitesimal lifting criterion for smoothness. By an argument similar to Proposition \ref{lem:complete} it is enough to consider the case when $X$ is an affine space $\Aa^m$. \\
    Let $R \onto T$ be an epimorphism of local Artin algebras over $\CC$. Consider a morphism $p:\Spec R \to \Strato{A}{X}$. We have to prove that there exists a diagonal morphism making the following diagram commute
    \[\begin{tikzcd}
	{\Spec T} & { \Strato{A}{X}} \\
	{\Spec R} & {\Spec \CC}
	\arrow["p",from=1-1, to=1-2]
	\arrow[from=1-1, to=2-1]
	\arrow[from=1-2, to=2-2]
	\arrow[dotted, from=2-1, to=1-2]
	\arrow[from=2-1, to=2-2]
\end{tikzcd}\]
The map $p$ is an element of $\StratF{A}{X}(\Spec T)$. By Proposition \ref{ap:trivial} any $\Spec A$--bundle over $\Spec T$ is trivial. Therefore, the map $p$ corresponds to an epimorphism
$$\varphi_p:T[x_1,\dots, x_m] \onto T\otimes_\CC A\,.$$
Pick any preimages of $\varphi_p(x_1),\dots,\varphi_p(x_m)$ in $R\otimes_\CC A$ and consider the induced map
$$R[x_1,\dots x_m] \to R\otimes_\CC A\,.$$
It is surjective due to Nakayama lemma. It corresponds to an element of $\StratF{A}{X}(\Spec R)$, so to a map from $\Spec R$ to $\Strato{A}{X}$ extending $p$.
\end{proof}
\begin{pro}
	Let $A$ be an algebra of rank $n$.
    We have
	$$
	\dim \Strato{A}{X}=n\dim X - \dim \Der(A,A)\,,
	$$
	where $\Der(A,A)$ is the vector space of derivations of $A$.
\end{pro}
\begin{proof}
    By Proposition \ref{lem:complete} we may assume that $X$ is an affine space $\Aa^m$.
	Pick a closed point $p_Z\in \Strato{A}{X}$ corresponding to a subscheme $Z \subset \Aa^{m}$. There is a short exact sequence
    $$
    \begin{tikzcd}
    0 & {\Der(A,A)} & {\Hom(\Omega_{\CC[x_1,\dots, x_m]},A)} & {T_{p_Z}\Hilb_n(X)}
    &{T^1_Z} 
    \arrow[from=1-1, to=1-2]
	\arrow[from=1-2, to=1-3]
	\arrow[from=1-3, to=1-4]
    \arrow["\psi",from=1-4, to=1-5]
    \end{tikzcd}
    $$
    The kernel of the map $\psi$ is the space of trivial embedded deformations of $Z$, cf. \cite[Sect. 1.2, p.2]{Schless}. By Corollary \ref{ap:tangent} it is the tangent space to $\Strato{A}{X}$ at $p_Z$. Therefore
	$$
    \dim T_Z\,\Strato{A}{X}=
    \dim \Hom(\Omega_{\CC[x_1,\dots, x_m]},A) - \dim \Der(A,A)=
	nm- \dim \Der(A,A)\,.
	$$ 
\end{proof}
\subsection{Proof of representability}
This section is devoted to the proof of Theorem \ref{ap:thm}.
Consider closed points of $\Hilb_n(X)$ that correspond to subschemes $Z$ such that \hbox{$\O_Z(Z)\not\simeq A$}, but $A$ deforms to $\O_Z(Z)$. We will prove that they determine a closed subset $T\subset \Hilb_n(X)$. Let $U$ its complement, we identify it with an open subfunctor of $\Hilb_n(X)$. We will prove that $\StratF{A}{X}$ is a closed subfunctor of $U$.
\begin{lemma}
    There exists a closed subset $T\subset \Hilb_n(X)$ such that a closed point $p_Z\in \Hilb_n$ lies in $T$ if and only if $\O_Z(Z)\not\simeq A$, but $A$ deforms to $\O_Z(Z)$.
\end{lemma}
\begin{proof}
    It suffices to show that these points define a closed subset on an open cover of $\Hilb_n(X)$. Let $W\subset \Hilb_n(X)$ be an open subset over which the tautological bundle is trivial. Trivialization of this bundle determine a choice of basis of the vector space $\O_Z(Z)$ for every point $p_Z\in W$. This yields a morphism
    $$
    \varphi:W\to \mathcal{M}_n\,.
    $$
    Then $T=\varphi^{-1}(\overline{O_A}\setminus O_A)$. Since $O_A$ is locally closed in $\mathcal{M}_n$, the subset $T$ is closed.
\end{proof}
\begin{lemma}
    The functor $\StratF{A}{X}$ is a subfunctor of $U$.
\end{lemma}
\begin{proof}
The functor $U$ is open in $\Hilb_n$, thus it is enough to check containment on closed points. It holds due to Proposition \ref{ap:trivial}.
\end{proof}
To prove that the inclusion $\StratF{A}{X}\subset U$ is a closed embedding it is enough to show that it satisfies the valuative criterion for properness. Let $R$ be a valuation ring. Denote by $\eta$ its generic point and by $x$ its closed point. Pick a map $f:\Spec R \to U$ such that $f(\eta)$ lies in $\StratF{A}{X}$. We have to prove that the map $f$ can be lifted to $\StratF{A}{X}$, that is, there exists a diagonal morphism making the following diagram commute
\[\begin{tikzcd}
	\eta & {\StratF{A}{X}} \\
	{\Spec R} & U
	\arrow[from=1-1, to=1-2]
	\arrow[from=1-1, to=2-1]
	\arrow[from=1-2, to=2-2]
	\arrow[dotted, from=2-1, to=1-2]
	\arrow["f", from=2-1, to=2-2]
\end{tikzcd}\]
 The map $f$ corresponds a family $\A$ over $\Spec R$. To show that it lifts to $\StratF{A}{X}$ we have to prove that the family $\A$ is a $\Spec A$--bundle. The family is flat over a local ring $R$, thus it corresponds to a free $R$-module. Therefore, it induces a map
$$
\psi_f:\Spec R \to \mathcal{M}_n
$$
\begin{lemma}
    Image of the map $\psi_f$ lies in the stratum $O_A$.
\end{lemma}
\begin{proof}
    The stratum $O_A$ is a locally closed subset of $\mathcal{M}_n$, cf. \cite[Thm. 2.14]{JS}. It can be written as $O_A=F\cap V$ where $F$ is closed, and $V$ is open. It is enough to show that the image is contained in $V$ and $F$ separately. \\
    The map $f$ sends the generic point $\eta$ to $\StratF{A}{X}$. By Proposition \ref{ap:trivial} we have $\psi_f(\eta)\in O_A$ and thus  $\psi_f(\eta)\in F$. The generic point is dense in $\Spec R$, so the whole image of $\psi_f$ lies in $F$. \\
    To prove that image lies in the open set $V$ it is enough to show that the closed point $x$ is mapped to it.
    The family $\A$ is a deformation of $A$ to $\A_{|x}$ over $R$. We have $f(x)\in U$, thus $\A_{|x}\simeq A$. It follows that $\psi_f(x)\in O_A \subset V$.
\end{proof}
The image of $\psi_f$ is contained in the stratum $O_A$, thus the family $\A$ is obtained as a pullback of $\mathcal{U}_{|O_A}\to O_A$. The family $\mathcal{U}_{|O_A}$ is a $\Spec A$--bundle, so $\A\to \Spec R$ is also a $\Spec A$--bundle. Therefore the map $f$ lifts to $\StratF{A}{X}$, the valuative criterion is satisfied and $\StratF{A}{X}$ is a closed subfunctor of $U$.
\begin{rem}
    To use the valuative criterion we have to know that the functor $\StratF{A}{X}$ is an algebraic space. It follows from \cite[Thm. 5.5.10]{Alper}.
\end{rem}

\bibliographystyle{alphaabbr}
\bibliography{refs}

\end{document}